\documentclass{amsart}
\usepackage{graphicx}
\usepackage{amsmath}
\usepackage{amsthm}
\usepackage{amsfonts}
\usepackage{dsfont}
\usepackage{amssymb}
\usepackage{microtype}
\usepackage[dvipsnames]{xcolor}
\usepackage{verbatim}
\usepackage{enumitem}
\usepackage{xfrac}

\makeatletter
\newtheorem*{rep@theorem}{\rep@title}
\newcommand{\newreptheorem}[2]{%
    \newenvironment{rep#1}[1]{%
        \def\rep@title{#2 \ref{##1}}%
        \begin{rep@theorem}
    }%
    {\end{rep@theorem}}
}
\makeatother

\theoremstyle{plain}
\newtheorem{thm}{Theorem}[section]
\newreptheorem{thm}{Theorem}
\newtheorem{lemma}{Lemma}[section]
\newtheorem{cor}[thm]{Corollary}
\newreptheorem{cor}{Corollary}
\newtheorem*{cor*}{Corollary}

\newtheorem*{thm*}{Theorem}

\theoremstyle{remark}
\newtheorem{rmk}[lemma]{Remark}

\newtheorem*{claim*}{Claim}

\theoremstyle{definition}

\newtheorem*{prob}{Problem}
\newtheorem{defn}[lemma]{Definition}

\numberwithin{equation}{section}

\def\E{\mathbb{E}}
\def\N{\mathbb{N}}
\def\P{\mathbb{P}}
\def\R{\mathbb{R}}
\def\Z{\mathbb{Z}}

\def\O{\mathbb{O}}

\def\1{\mathds{1}}
\def\0{\mathbf{0}}

\def\cA{\mathcal{A}}

\def\cC{\mathcal{C}}

\def\cS{\mathcal{S}}

\def\cL{\mathcal{L}}

\def\cT{\mathcal{T}}
\def\cO{\mathcal{O}}

\def\cP{\mathcal{P}}
\def\cR{\mathcal{R}}
\def\cI{\mathcal{I}}

\def\MS{\mathcal{MS}}
\def\RS{\mathcal{RS}}

\def\bF{\mathbf{F}}

\DeclareMathOperator{\Var}{Var}
\DeclareMathOperator{\Cov}{Cov}
\DeclareMathOperator{\Sym}{Sym}
\DeclareMathOperator{\opint}{int}
\DeclareMathOperator{\ord}{ord}
\DeclareMathOperator{\diam}{diam}

\newcommand{\dquote}[1]{``{#1}''}

\newcommand{\tst}{\text{ s.t. }}
\newcommand{\tand}{\text{ and }}

\definecolor{darkerred}{RGB}{192,0,0}
\definecolor{darkerblue}{RGB}{0,0,160}
\definecolor{darkgreen}{RGB}{0,160,0}

\title{Shotgun Identification on Groups}

\author{Jacob Raymond}
\address{Jacob Raymond\\
Department of Mathematics and Statistics\\
University of North Carolina at Charlotte \\
9201 University City Blvd.\\
Charlotte, NC 28223}
\email{jraymon9@uncc.edu}

\author{Robert Bland}
\address{Robert Bland\\
Department of Mathematics and Statistics\\
University of North Carolina at Charlotte \\
9201 University City Blvd.\\
Charlotte, NC 28223}
\email{rbland5@uncc.edu}

\author{Kevin McGoff}
\address{Kevin McGoff\\
Department of Mathematics and Statistics\\
University of North Carolina at Charlotte \\
9201 University City Blvd.\\
Charlotte, NC 28223}
\email{kmcgoff1@uncc.edu}
\urladdr{https://clas-math.uncc.edu/kevin-mcgoff/}


\begin{document}

\maketitle

\begin{abstract}
    We consider the problem of shotgun identification of patterns on groups, which extends previous work on shotgun identification of DNA sequences and labeled graphs. A shotgun identification problem on a group $G$ is specified by two finite subsets $C \subset G$ and $K \subset G$ and a finite alphabet $\mathcal{A}$. In such problems, there is a ``global" pattern $w \in \mathcal{A}^{CK}$, and one would like to be able to identify this pattern (up to translation) based only on observation of the ``local" $K$-shaped subpatterns of $w$, called reads, centered at the elements of $C$. We consider an asymptotic regime in which the size of $w$ tends to infinity and the symbols of $w$ are drawn in an \textit{i.i.d.}~fashion. Our first general result establishes sufficient conditions under which the random pattern $w$ is identifiable from its reads with probability tending to one, and our second general result establishes sufficient conditions under which the random pattern $w$ is non-identifiable with probability tending to one. Additionally, we illustrate our main results by applying them to several families of examples.
\end{abstract}

%%%%%%%%%%%%%%%%%%%%%%%%%%%%%%%%%%%%%%%%%%
%%%%%%%%%%%%%%%%%%%%%%%%%%%%%%%%%%%%%%%%%%
\section{Introduction}  \label{Sect:Intro}
%%%%%%%%%%%%%%%%%%%%%%%%%%%%%%%%%%%%%%%%%%
%%%%%%%%%%%%%%%%%%%%%%%%%%%%%%%%%%%%%%%%%%

The basic task in any \dquote{shotgun} assembly or identification problem is to reconstruct some type of global object using only knowledge of its local structure. As a prototypical problem of this type, one may imagine trying to reconstruct a jigsaw puzzle without knowledge of the image it creates. The term \dquote{shotgun} comes from the problem of genome sequencing, in which one attempts to reconstruct an organism's genome by connecting relatively small fragments of the organism's DNA, called reads. Another example of this type of problem that has been treated in the literature is the reconstruction of a labeled graph from observations of the neighborhoods within the graph.

In this paper, we are interested in shotgun identification problems on groups. 
%Informally, the group structure provides an orientation for each of the local reads. 
The group structure can be thought of as providing an orientation for each of the reads.
To describe these problems formally, we begin with a countable (finite or infinite) group $G$, a finite \textit{center set} $C \subset G$, and a finite \textit{read shape} $K \subset G$. 

Additionally, we let $\cA$ be a finite set of symbols, called the alphabet.
We suppose that there is a \textit{pattern} $w \in \cA^{CK}$, and in principle we would like to reconstruct $w$. However, we do not have access to the entire pattern. Instead, we only observe the $K$-shaped patterns $w|_{cK}$ centered at the elements $c$ in $C$.
More precisely, we let $\MS(E)$ denote the set of multi-sets of any set $E$, and we define the \textit{read operator} $\cR_{C,K}:\cA^{CK} \to \MS(\cA^K)$ by setting
\begin{equation*}
    \cR_{C,K}(w) = \bigl\{w|_{cK} : c \in C\bigr\} \subset \cA^K.
\end{equation*}
The elements in $\cR_{C,K}(w)$ are known as \textit{reads}. We remark that we view the reads as patterns on $K$, as opposed to patterns on $cK$, so that the location of these subpatterns within the original pattern $w$ is not known. For more details, see Section \ref{sect:background}. With these definitions, the goal of shotgun reconstruction or assembly is to reconstruct the pattern $w \in \cA^{CK}$ based only on knowledge of $G$, $C$, $K$, $\cA$, and $\cR_{C,K}(w)$. 

The idea of identifiability, defined formally below, is that it describes the situation when the reads $\cR_{C,K}(w)$ contain enough information to identify the original pattern $w$. However, first we must carefully define what we consider to be successful identification. In the group-theoretic setting considered here, we need to account for the presence of symmetries arising from group translations. Let us say that a pattern $w' \in \cA^{CK}$ is a \textit{$C$-preserving translate} of $w$ if there exists a group element $g \in G$ that preserves the center set (i.e., $gC = C$) and such that $w(h) = w'(gh)$ for all $h \in CK$. If $w'$ is such a translate of $w$, then $\cR_{C,K}(w) = \cR_{C,K}(w')$ (see Lemma \ref{id_char}), and therefore it is impossible to distinguish $w$ and $w'$ from their reads. Thus, we consider reconstruction to be successful if the reconstructed pattern is a $C$-preserving translate of the original pattern.

For a pattern $w \in \cA^{CK}$, it may happen that $\cR_{C,K}(w)$ provides enough information to determine $w$ (up to translation), and in other cases there may be (many) other patterns in $\cA^{CK}$ that are consistent with the observed reads. In the present setting, we consider $w$ to be $(C,K)$-\textit{identifiable} (or just identifiable) if $\cR_{C,K}(w)$ provides enough information to determine $w$ uniquely up to $C$-preserving translation. More precisely, we say that a pattern $w \in \cA^{CK}$ is $(C,K)$-identifiable if the following condition holds: for every pattern $w' \in \cA^{CK}$, if $\cR_{C,K}(w') = \cR_{C,K}(w)$, then there exists $g \in G$ such that $gC = C$ and $w(h) = w'(gh)$ for every $h \in CK$. We remark that there is a naive reconstruction algorithm (exhaustive search for a pattern that is consistent with the reads, with ties broken arbitrarily) such that for all identifiable $w$, the algorithm returns a $C$-preserving translate of $w$ when given $\cR_{C,K}(w)$ as input. Furthermore, there is \textit{no} deterministic reconstruction algorithm such that for all non-identifiable patterns $w$, the algorithm returns a $C$-preserving translate of $w$ when given $\cR_{C,K}(w)$ as input. 

In this work we seek to analyze the probability of a random pattern being identifiable. To state our main results, we require a few more definitions. First, let $\cI_{C,K}$ be the set of all $(C,K)$-identifiable patterns in $\cA^{CK}$, and then define a probability space on $\cA^{CK}$ with the discrete $\sigma$-algebra and the product measure $\mu_p = p^{CK}$ for some probability vector $p = (p_a)_{a \in \cA}$ on $\cA$. In order to avoid trivialities, we assume throughout that $p$ gives positive probability to at least two symbols in $\cA$. 

We can then consider the measure of all identifiable patterns $\mu_p(\cI_{C,K})$, i.e., the probability that a random pattern is identifiable. In fact, we consider the asymptotic regime in which the size of the problem tends to infinity. That is, suppose we are given sequences $\{G_n\}$, $\{C_n\}$, and $\{K_n\}$ such that $|C_n K_n| \to \infty$, as well as some fixed $\cA$ and $p$. Let $\{\mu^n_p\}$ be the sequence of probability measures on $\cA^{C_n K_n}$ defined by $\mu^n_p = p^{C_nK_n}$. Then we say that identifiability holds \textit{asymptotically almost surely} (a.a.s.) if
\begin{equation*}
    \lim_{n \to \infty} \mu^n_p\bigl(\cI_{C_n,K_n}\bigr) = 1,
\end{equation*}
and non-identifiability holds a.a.s. if %\revision{}{(added , in equation below)}
\begin{equation*}
    \lim_{n \to \infty} \mu^n_p\bigl(\cI_{C_n,K_n}\bigr) = 0.
\end{equation*}

%%%%%%%%%%%%%%%%%%%%%%%%%%%%%%%%%%%%%%
\subsection{Statement of main results}
%%%%%%%%%%%%%%%%%%%%%%%%%%%%%%%%%%%%%%

Our main results address both identifiability and non-identifiability. In our first main result (Theorem \ref{pos:regime} below), we give sufficient conditions to have identifiability a.a.s. In our second main result (Theorem \ref{neg:regime} below), we give sufficient conditions to have non-identifiability a.a.s. In Section \ref{Sect:ExamplesIntro}, we also present applications of these results to some specific families of examples, including cases in which the read shapes are hypercubes in the group $\Z^d$ or balls (in the word metric) %\revision{around the origin}{(remove or change to \dquote{... in which the read shapes are balls around the origin}, either would work)} 
in finitely generated groups.

In order to state our general identifiability result, we require a few more definitions. Along with the countable group $G$ and finite subsets $C$ and $K$, we let $\bF \subset \cP(K)$ (with $\cP(K)$ being the power set of $K$), and we define the \textit{overlap graph} $\cO(G,C,K,\bF)$ as follows. The vertices of this graph are the elements of $C$, and two elements $c_1,c_2 \in C$ are connected by an edge if there exists $g \in G$ and $F \in \bF$ such that
\begin{equation*}
    gF \subset c_1K \cap c_2K.
\end{equation*}
Additionally, the second order R\'enyi entropy $H_2(p)$ is defined as
\begin{equation*}
    H_2(p) = -\ln\left(\sum_{a \in \cA} p_a^2\right).
\end{equation*}
We remark that as a result of our nontriviality assumption on $p$, we always have $H_2(p) > 0$. Additionally, we note that the choice of the natural logarithm here is arbitrary. In fact, Theorems \ref{pos:regime} and \ref{neg:regime} are independent of the choice of base, since their statements only ever refer to ratios of logarithms.

We consider the objects $\cA$, $p$, $\{G_n\}$, $\{C_n\}$, $\{K_n\}$, and $\{\mu_p^n\}$ to be defined as in the previous section. Since our primary goal is to consider the asymptotic regime in which the size of the pattern to be reconstructed tends to infinity, we require that
\begin{equation*} \tag{G1} \label{assump:CK_to_inf}
    \lim_{n \to \infty} |C_nK_n| = \infty.
\end{equation*}

For our general identifiability result, we also require the existence of a sequence of sets of overlap shapes $\{\bF_n\}$, with $\bF_n \subset \cP(K_n)$, satisfying the following assumptions. The first identifiability condition is
\begin{equation*} \tag{I1} \label{assump:overlap_conn}
    \forall n \gg 0, \quad \cO(G_n, C_n, K_n, \bF_n) \text{ is connected, }
\end{equation*}
where $\forall n \gg 0$ denotes \dquote{for all large enough $n$.} This condition ensures that the reads can be organized into a single connected component based on their $\bF_n$-overlaps. Our next identifiability condition is
\begin{equation*} \tag{I2} \label{assump:few_oshapes}
    \lim_{n \to \infty} \frac{\ln(|\bF_n|)}{\ln(|C_nK_n|)} = 0.
\end{equation*}
We remark $|\bF_n|$ denotes the number of sets in the collection $\bF_n$, i.e., the number of overlap shapes. Although the sequence of overlap shapes can increase in size to infinity, condition \eqref{assump:few_oshapes} guarantees that this growth is controlled. This assumption is discussed in further detail in Section \ref{sect:pos}. Our final identifiability condition is
\begin{equation*} \tag{I3} \label{assump:oshape_size_bound}
    \exists \epsilon > 0, \forall n \gg 0, \forall F \in \bF_n, \quad |F| \ge (1 + \epsilon)\frac{2}{H_2(p)}\ln\bigl(|C_nK_n|\bigr).
\end{equation*}
This condition (together with condition (\ref{assump:overlap_conn})) guarantees that there are sufficiently large overlapping regions between the reads. 
We may now state our main identifiability result, which we prove in Section \ref{sect:pos}.
\begin{thm} \label{pos:regime}
    Let $\cA$, $p$, $\{G_n\}$, $\{C_n\}$, $\{K_n\}$, and $\{\mu_p^n\}$ be as above, and suppose that \eqref{assump:CK_to_inf} is satisfied. If there exists a sequence $\{\bF_n\}$ such that $\bF_n \subset \cP(K_n)$ and conditions \eqref{assump:overlap_conn}, \eqref{assump:few_oshapes}, and \eqref{assump:oshape_size_bound} are satisfied, then identifiability holds a.a.s.
\end{thm}

Informally, this theorem states that if the collection of read positions $\{cK:c \in C\}$ has sufficiently large and regular overlapping regions, then the probability of drawing an identifiable pattern is large (tending to one asymptotically).

We now turn our attention to the statement of our main non-identifiability result. The basis for this result is that if 
%a pattern has a repeated subpattern with a particular shape at two different locations,
a pattern has a certain type of subpattern called a blocking configuration,
then it is impossible to determine the original pattern from its reads, i.e., the pattern is non-identifiable. In this work we focus on blocking configurations that arise from repeated shapes called \textit{shells}. A shell around an element $a \in CK$ is defined as the set $(C \cap aK^{-1})K \setminus \{a\}$. One should imagine this set as a generalization of a symmetric hypercube in $\Z^d$ with the center removed (hence the term \dquote{shell}). We also say the \textit{shell type} at $a$ is the set $a^{-1}C \cap K^{-1}$; see Section \ref{sect:neg} for a more detailed discussion of shells and shell types. For any $E \subset CK$, let $I_E$ be the collection of shell types that appear for elements $a \in E$, and note that $|I_E|$ is the number of shell types associated to $E$.

We consider the objects $\cA$, $p$, $\{G_n\}$, $\{C_n\}$, $\{K_n\}$, and $\{\mu_p^n\}$ as in the previous section. As before, we are interested in the asymptotic regime in which \eqref{assump:CK_to_inf} holds. For our general non-identifiability result, we require the existence of a sequence $\{B_n\}$, with $B_n \subset C_nK_n$, satisfying the following assumptions. The first non-identifiability condition is
\begin{equation*} \tag{N1} \label{assump:B_few_shells}
    \lim_{n \to \infty} \frac{\ln(|I_{B_n}|)}{\ln(|C_nK_n|)} = 0.
\end{equation*}
This condition asserts that the number of shell types represented in $B_n$ cannot grow too quickly. Our second non-identifiability condition is
\begin{equation*} \tag{N2} \label{assump:B_large}
    \lim_{n \to \infty} \frac{\ln(|B_n|)}{\ln(|C_nK_n|)} = 1.
\end{equation*}
This condition ensures that $B_n$ is sufficiently large. For our last condition, we require
\begin{equation*} \tag{N3} \label{assump:shell_size_bound}
    \exists \epsilon > 0, \forall n \gg 0, \quad |K_n^{-1}K_n| \le (1 - \epsilon)\frac{2}{H_2(p)}\ln\bigl(|C_nK_n|\bigr).
\end{equation*}
This condition guarantees that all shells (which must be contained in a translate of $K_n^{-1} K_n$) are sufficiently small.
The following theorem is our main non-identifiability result. Its proof appears in Section \ref{sect:neg}.
\begin{thm} \label{neg:regime}
    Let $\cA$, $p$, $\{G_n\}$, $\{C_n\}$, $\{K_n\}$, and $\{\mu_p^n\}$ be as above such that \eqref{assump:CK_to_inf} is satisfied. Suppose there exists a sequence $\{B_n\}$ with $B_n \subset C_nK_n$ such that \eqref{assump:B_few_shells} and \eqref{assump:B_large} are satisfied. Further, suppose \eqref{assump:shell_size_bound}. Then non-identifiability holds a.a.s.
\end{thm}

This result states that if the read shape is small enough and there is a large enough set of positions in $C_nK_n$ with relatively few associated shell types, then the probability of drawing an identifiable pattern is low (tending to zero asymptotically). 

\begin{rmk}
    The appearance of conditions \eqref{assump:oshape_size_bound} and \eqref{assump:shell_size_bound} in our results suggests that there may be some critical phenomenon that arises for shotgun identification problems on groups. In particular, consider the constant
    \begin{equation*}
         \lambda_c = \frac{2}{H_2(p)}.
    \end{equation*}
    Then condition \eqref{assump:oshape_size_bound} ensures that $|K_n|/ \ln(|C_n K_n|) > \lambda_c$, whereas condition \eqref{assump:shell_size_bound} guarantees that $|K_n|/ \ln(|C_n K_n|) < \lambda_c$. Thus, our results suggest that (under appropriate regularity conditions) $\lambda_c$ might serve as a kind of critical value for the ratio $|K_n|/ \ln( |C_n K_n|)$. In other words, our results suggest the heuristic that if the ratio $|K_n|/ \ln( |C_n K_n|)$ is less than $\lambda_c$, then non-identifiability occurs with high probability, and if the ratio is greater than $\lambda_c$, then identifiability occurs with high probability. This phenomenon has been observed previously for $G = \Z$ \cite{motahari} and conjectured for $G  = \Z^d$ \cite{mossel}. See Sections \ref{Sect:LatticesBasic} and \ref{sect:conc} for additional discussion. 
\end{rmk}

%%%%%%%%%%%%%%%%%%%%%%%%%%%%%%%%%
\subsection{Examples} \label{Sect:ExamplesIntro}
%%%%%%%%%%%%%%%%%%%%%%%%%%%%%%%%%

Here we describe several applications of our main results to illustrate how they may be used in various settings. The proofs of these statements appear in Section \ref{sect:ex}.

\subsubsection{Hypercubes in $\Z^d$ with sparse center sets} \label{sect:ex1}

In our first example, we consider the $d$-dimensional integer lattice $\Z^d$. When $d = 1$, one may imagine performing reconstruction on a string of characters, as is the case for DNA sequencing. When $d = 2$, one may imagine trying to reconstruct a pattern on a grid of points in the plane, as might be the case with images. We present our results for this example with arbitrary $d \in \N$. We let $\{G_n\}$ be defined as $G_n = (\Z^d, +)$. We then consider three sequences of natural numbers $\{m_n\}$,$\{r_n\}$, and $\{\ell_n\}$, where $\{m_n\}$ tends to infinity and $\{\ell_n/r_n\}$ tends to zero. The sequence of center sets $\{C_n\}$ is then defined by setting $C_n = \ell_n \cdot [0, m_n]^d \subset \Z^d$, i.e., each $C_n$ is a regular grid of $m_n^d$ points spaced $\ell_n$ units apart. 

With $\{K_n\}$ defined as $K_n = [0,r_n-1]^d \subset \Z^d$ (so that the reads are hypercubes in $\Z^d$ of side length $r_n$), we let $R_n = m_n\ell_n + r_n$, which results in $C_n K_n$ being a hypercube of side length $R_n$. In this context, we are mainly interested in the question, how large must the reads be (parameterized by $r_n$) in terms of the observable region (parameterized by $R_n$) in order to have identifiability with high probability? The following corollary of Theorems \ref{pos:regime} and \ref{neg:regime} gives our identifiability and non-identifiability results for this example. 
\begin{cor} \label{ex1}
    If there exists some $\epsilon > 0$ such that for all large $n$,
    \begin{equation*}
        r_n^d \ge (1 + \epsilon) \frac{2d}{H_2(p)}\ln(R_n),
    \end{equation*}
    then identifiability occurs a.a.s. On the other hand, if there exists some $\epsilon > 0$ such that for all large $n$,
    \begin{equation*}
        r_n^d \le (1 - \epsilon) \frac{d}{2^{d-1}H_2(p)}\ln(R_n),
    \end{equation*}
    then non-identifiability occurs a.a.s.
\end{cor}
In the special case where $\ell_n = 1$, this corollary recovers existing results from previous work \cite{mossel,motahari}. The case $\ell_n >1$ may be viewed as allowing some sparsity in the observations. See Section \ref{Sect:LatticesBasic} for additional discussion.

\subsubsection{General read shapes in $\Z^d$} \label{sect:ex2}

Next we consider general read shapes in $\Z^d$. We take the center set $C_n$ to be the hypercube $[0,n-1]^d \subset \Z^d$, and we seek conditions on the sequence of read shapes $\{K_n\}$ that guarantee identifiability or non-identifiability a.a.s.

In order to state our identifiability result, we first define the \textit{$1$-interior} of $F \subset G$, denoted by $\opint_1 F$, by setting
\begin{equation*}
    \opint_1 F = \bigcap_u \bigl(F \cap u F \cap u^{-1} F \bigr),
\end{equation*}
where the intersection ranges over the standard basis vectors $u$ in $\Z^d$. For our non-identifiability result, we define the \textit{diameter} of $F \subset G$, denoted by $\diam(F)$, by setting
\begin{equation*}
    \diam(F) = \sup \bigl\{ \| a^{-1}b \|_{\infty} : a,b \in F \bigr\},
\end{equation*}
where $\| \cdot \|_{\infty}$ denotes the standard $\ell^\infty$ norm on $\R^d$. We now state our results for this example.
\begin{cor} \label{ex2}
    If 
    \begin{equation*}
       \lim_{n \to \infty} \frac{|\opint_1K_n|}{|K_n|} = 1,
    \end{equation*}
    and for all large $n$,
    \begin{equation*}
        |K_n| \ge \frac{4d}{H_2(p)}\ln(n),
    \end{equation*}
    then identifiability occurs a.a.s. On the other hand, if
    \begin{equation*}
        \lim_{n \to \infty} \frac{\diam(K_n)}{n} = 0,
    \end{equation*}
    and for some $\epsilon > 0$ and all large $n$,
    \begin{equation*}
        |K_n^{-1}K_n| \le (1 - \epsilon) \frac{2d}{H_2(p)} \ln(n),
    \end{equation*}
    then non-identifiability occurs a.a.s.
\end{cor}
This corollary gives that if the read shape is sufficiently large and has relatively large 1-interior, then identifiability occurs with high probability. On the other hand, if the read shape is sufficiently small and not too spread out 
(in the sense that its diameter is small enough), then non-identifiability occurs with high probability.

\subsubsection{Finitely generated groups} \label{sect:ex3}

In our last family of examples, we consider the setting of finitely generated groups. Let $T$ be a finite symmetric set of generators of $G$, and define $T_i$ to be the closed ball of radius $i$ around the identity with respect to the standard word metric defined by $T$, i.e., $T_i$ is the set of elements $g \in G$ such that $g = t_1 \dots t_j$ for some $j \le i$ and $t_1,\dots,t_j \in T$. Let $\gamma(n) = |T_n|$ be the growth rate function of the group $G$ with generating set $T$, and consider sequences of natural numbers $\{R_n\}$ and $\{r_n\}$ with $R_n > r_n$. We take the center set $C_n$ to be  $T_{R_n - r_n}$, and we take the read shape $K_n$ to be $T_{r_n}$, which makes the observable region $C_nK_n = T_{R_n}$. In this setting, we would like to understand how the read radius $r_n$ and the observable radius $R_n$ affect the probability of identifiability. Interestingly, the answer depends on the growth rate function of the group.

\begin{cor} \label{ex3}
    If there exists some $\epsilon > 0$ such that for all large $n$,
    \begin{equation*}
        \gamma(r_n-1) \ge (1 + \epsilon) \frac{2}{H_2(p)}\ln(\gamma(R_n)),
    \end{equation*}
    then identifiability occurs a.a.s. On the other hand, if there exists some $\epsilon > 0$ such that for all large $n$,
    \begin{equation*}
        \gamma(2r_n) \le (1-\epsilon) \frac{2}{H_2(p)}\ln(\gamma(R_n)),
    \end{equation*}
    then non-identifiability occurs a.a.s.
\end{cor}
In short, if the read radius $r_n$ is large enough relative to the observable radius $R_n$ (in the sense that the first inequality in corollary is satisfied), then we have identifiability with high probability. However, if the read radius $r_n$ is sufficiently small relative to the observable radius $R_n$ (in the sense that the second inequality in the corollary is satisfied), then non-identifiability occurs with high probability.

\begin{rmk}%\revision{}{(made outer parenthesis in math below big)}
    For a finitely generated group $G$ with polynomial growth of degree $d$, the previous corollary gives that if $r_n = \omega\bigl(\sqrt[d]{\ln(R_n)}\bigr)$, then identifiability holds a.a.s., and if $r_n = o\bigl(\sqrt[d]{\ln(R_n)}\bigr)$, then non-identifiability holds a.a.s. Similarly, for a finitely generated group $G$ with exponential growth, the previous corollary gives that if $r_n = \omega\bigl(\ln(R_n)\bigr)$, then identifiability occurs a.a.s., and if $r_n = o\bigl(\ln(R_n)\bigr)$, then non-identifiability occurs a.a.s. Additionally, if one has more precise information about the growth rate function of the group $G$, then the previous corollary may be used to deduce sharper results about identifiability and non-identifiability.
\end{rmk}

%%%%%%%%%%%%%%%%%%%%%%%%%%%%%%%%%%%%%%%%%%%%%%%
\subsection{Related work} \label{ssect:relwork}
%%%%%%%%%%%%%%%%%%%%%%%%%%%%%%%%%%%%%%%%%%%%%%%

The approach to DNA sequencing now known as shotgun sequencing has been used in genomics for decades. See the review \cite{green2001strategies} for a detailed description of this methodology from the biological perspective, along with original references. In complementary work, the mathematical analysis of shotgun sequencing has also been studied over many years. For early work in this direction, see \cite{lander1988genomic}. For more recent results, see \cite{motahari} and references therein.

The mathematical problem of shotgun assembly of DNA sequences was generalized substantially by Mossel and Ross \cite{mossel}. In their formulation, one observes the $r$-neighborhoods of the vertices of a labeled graph (for some radius $r$) and attempts to reconstruct the graph (up to an appropriate notion of isomorphism). The case of DNA sequences can be recovered by considering line graphs. They established detailed results for certain classes of graphs, including  the natural (nearest-neighbor) graph associated to $d$-dimensional integer lattices $\Z^d$. We discuss this work in more detail and compare it with our results in Section \ref{Sect:LatticesBasic}. More recent work in this direction has considered shotgun assembly of random graphs \cite{mossel2015shotgun} and shotgun assembly of random jigsaw puzzles \cite{bordenave2020shotgun}.

Substantial attention has been devoted to understanding specific algorithms for shotgun DNA sequencing; see the reviews \cite{miller2010assembly,pop2009genome}. Recent work on jigsaw percolation \cite{bollobas2015threshold,brummitt2015jigsaw,gravner2017nucleation} provides expected time complexity bounds for solving graph-based jigsaw puzzles using a randomized algorithm. By viewing the graph-based jigsaw puzzle as a shotgun reconstruction problem, the results on jigsaw percolation can be seen to provide an analysis of certain shotgun reconstruction algorithms on graphs. Although we do not pursue algorithmic aspects of shotgun reconstruction on groups in this work, we view it as an interesting avenue for future work.

Finally, we note that the shotgun identification problem on groups is related to a topic known as synchronization problems over groups. In such problems, one is given an index set $S$ and a collection of some group elements $g_{u,v}$, with $u,v \in S$, and one would like to estimate a collection of group elements $(g_u)_{u \in S}$ such that $g_{u,v} = g_u g_v^{-1}$. These problems have been considered from a wide variety perspectives, including applied \cite{perry2018message,singer2011three}, geometric \cite{gao2019geometry}, and probabilistic \cite{abbe2017group}. Although we do not address synchronization problems directly in this work, we note that one may derive a type of synchronization problem from the shotgun identification problem: let $S$ be a collection of reads, and let $g_{u,v}$ be defined as the group translation of read $u$ that has maximal overlap with read $v$. In future work it  might be interesting to analyze the relationship between these problems more closely. 

%%%%%%%%%%%%%%%%%%%%%
\subsection{Overview}
%%%%%%%%%%%%%%%%%%%%%

In Section \ref{sect:background}, we present the necessary background and notation. Section \ref{sect:prob} contains statements and proofs of the basic probabilistic tools that we use to prove our main results. In Section \ref{sect:pos} we prove our main identifiability result, Theorem \ref{pos:regime}, and in Section \ref{sect:neg} we prove our main non-identifiability result, Theorem \ref{neg:regime}. In Section \ref{sect:ex}, we apply these theorems to some families of examples. Finally, in Section \ref{sect:conc} we discuss possible directions for future research. 

%%%%%%%%%%%%%%%%%%%%%%%%%%%%%%%%%%%%%%%%%%%%%%%%%%%%%%%%%
%%%%%%%%%%%%%%%%%%%%%%%%%%%%%%%%%%%%%%%%%%%%%%%%%%%%%%%%%
\section{Background and notation} \label{sect:background}
%%%%%%%%%%%%%%%%%%%%%%%%%%%%%%%%%%%%%%%%%%%%%%%%%%%%%%%%%
%%%%%%%%%%%%%%%%%%%%%%%%%%%%%%%%%%%%%%%%%%%%%%%%%%%%%%%%%

%%%%%%%%%%%%%%%%%%%%%%%%%%%%%%%%%%%%%
\subsection{Basic notation}
%%%%%%%%%%%%%%%%%%%%%%%%%%%%%%%%%%%%%

Here and throughout the paper we assume that $G$ is a countable group, $\cA$ is a finite set, called the alphabet, $C$ is a finite nonempty subset of $G$, and $K$ is a finite nonempty subset of $G$. We endow $\cA$ with the discrete topology and $\cA^G$ with the product topology, which makes $\cA^G$ into a compact metrizable space. Then we endow $\cA^G$ with the Borel $\sigma$-algebra and suppress this choice in our notation.

For $g \in G$ and $A \subset G$, we let $gA = \{g a : a \in A\}$. Additionally, we say that two subsets $A$ and $B$ of $G$ are translates (of each other) if there is some $g \in G$ such that $A = gB$. Further, we may also \dquote{invert} a subset with the notation $A^{-1} = \{ a^{-1} : a \in A\}$, and for two subsets $A,B \subset G$, we let $AB = \{ab : a \in A, b \in B\}$. We let $|E|$ denote the cardinality of any finite set $E$, and note that for any $g \in G$ and $A \subset G$, we have $|gA| = |A|$. Lastly, for any element $g \in G$, we let $\ord(g)$ denote the order of the element $g$, and we let $\langle g \rangle$ denote the cyclic group generated by $g$.

%%%%%%%%%%%%%%%%%%%%%%%%%%%%%%%%%%%%%%%%%%%%%%%%%
\subsection{Patterns, reads, and identifiability}
%%%%%%%%%%%%%%%%%%%%%%%%%%%%%%%%%%%%%%%%%%%%%%%%%

We now introduce the notions of patterns, reads, and identifiability, all of which are used throughout the remainder of the paper.

\begin{defn}
    Let $F \subset G$. Any element $w \in \cA^F$ is called a \textit{pattern}, and we may say that the pattern $w \in \cA^F$ has \textit{shape} $F$.
\end{defn}

We view patterns with shape $F$ as a function from $F$ to $\cA$. For ease of notation, if $w \in \cA^F$ and $E \subset F$, then we let $w(E) = w|_E$ denote the restriction of $w$ to the shape $E$. For any finite shape $F \subset G$ and any pattern $w$ with shape $F$, we define the associated cylinder set,
\begin{equation*}
    [w] = \Bigl\{ x \in \cA^G : x(F) = w \Bigr\}.
\end{equation*}
Any such cylinder set is a clopen subset of $\cA^G$, and we remind the reader that the clopen subsets of $\cA^G$ generate the product topology and the Borel $\sigma$-algebra on $\cA^G$.

For any $F \subset G$, pattern $w \in \cA^F$, and $g \in G$, we define the function $\sigma^g:\cA^F \to \cA^{g^{-1}F}$ by the rule
\begin{equation*}
     \sigma^{g}(w)(h) = w(gh), \quad \forall h \in g^{-1}F.
\end{equation*}
We remark that for any $g \in G$, the map $\sigma^g : \cA^G \to \cA^G$ is a homeomorphism. Also, we note that the collection $\{\sigma^g\}_{g \in G}$ defines a \textit{right} group action on $\cA^G$, meaning that it satisfies the composition rule $\sigma^g \circ \sigma^h = \sigma^{hg}$ (as opposed to the more standard composition rule $\sigma^g \circ \sigma^h = \sigma^{gh}$ of left group actions). This right group action is equivalent to the usual (left) group action of the opposite group $G^\text{op}$. However, in this work we do not use any properties of $\sigma$ as a group action; instead, we simply use this notation to refer to various translates of patterns. 

For any finite set $F$, we let $\MS(F)$ denote the set of all multi-sets of elements of $F$. We now give a proper definition of the read operator.
\begin{defn}
    For any $w \in \cA^{CK}$, we define the multi-set
    \begin{equation*}
        \cR_{C,K}(w) = \bigl\{ \sigma^c(w)(K) : c \in C \bigr\}.
    \end{equation*} 
    Note that $\cR_{C,K}$ defines a map $\cR_{C,K} : \cA^{CK} \to \MS(\cA^K)$, which we call the \textit{read operator}. Furthermore, any element of $\cR_{C,K}(w)$ is called a \textit{read}.
\end{defn}

Next we say precisely what it means for a pattern to be identifiable.
\begin{defn}
    Let $w \in \cA^{CK}$. The \textit{identifiability class} of $w$, denoted $\cC(w)$, is the set of patterns that are equal to $w$ after a group translation that preserves the set of centers:
    \begin{equation*}
        \cC(w) = \Bigl\{ v \in \cA^{CK} : \exists g \in G \text{ such that } gC = C \tand \sigma^g (v) = w \Bigr\}.
    \end{equation*}
    Furthermore, a pattern $w \in \cA^{CK}$ is $(C,K)$-\textit{identifiable} (or just \textit{identifiable}) if 
    \begin{equation*}
        \cR_{C,K}^{-1} \bigl( \cR_{C,K}(w) \bigr) = \cC(w),
    \end{equation*}
    where $\cR_{C,K} : \cA^{CK} \to \MS(\cA^K)$ is the read operator and $\cR_{C,K}^{-1}$ is its set inverse function. 
\end{defn}

In other words, $w$ is identifiable if the only other patterns from $\cA^{CK}$ with the same multi-set of reads are translations of $w$ by an element that preserves the set of centers.
The following lemma provides a useful characterization of when $\cR_{C,K}(w) = \cR_{C,K}(w')$ for two patterns $w,w' \in \cA^{CK}$.

\begin{lemma} \label{read_eq}
    Let $G$ be a group, let $C,K \subset G$, and let $w, w' \in \cA^{CK}$. Then $\cR_{C,K}(w) = \cR_{C,K}(w')$ if and only if there exists a permutation $\phi: C \to C$ such that $\forall c \in C$,
    \begin{equation*}
        \sigma^c(w)(K) = \sigma^{\phi(c)}(w')(K).
    \end{equation*}
    \begin{proof}
        Assume $\cR_{C,K}(w) = \cR_{C,K}(w')$. For $u \in \cA^K$, let $n_u$ denote the number of times that $u$ appears in $\cR_{C,K}(w)$. Then define three sets:
        \begin{align*}
            S_1 & = \{ (u,c) : u \in \cA^K, c \in C, \sigma^c(w)(K) = u\} \\
            S_2 & = \{ (u,i) : u \in \cA^K, 1 \leq i \leq n_u \} \\
            S_3 & = \{ (u,c) : u \in \cA^K, c \in C, \sigma^c(w')(K) = u\}.
        \end{align*}
        By definition of $\cR_{C,K}(w)$, there is a bijection $\phi_1 : S_1 \to S_2$, and similarly there is a bijection $\phi_2 : S_3 \to S_2$. Additionally, since each center gives rise to exactly one element of $S_1$ and one element of $S_3$ there are bijections $\phi_3 : C \to S_1$ and $\phi_4 : C \to S_3$. Then the desired permutation is given by $\phi : C \to C$, where $\phi = \phi_4^{-1}\circ\phi_2^{-1} \circ \phi_1 \circ \phi_3$.
        
        The reverse implication is clear from the definitions.
    \end{proof}
\end{lemma}

With this characterization, we can prove that one of the set containments required for identifiability is always true. Indeed, any pattern in the identifiability class of $w$ must have the same multi-set of reads as $w$. We record this fact in the following lemma. 
\begin{lemma} \label{id_char}
    Suppose $G$ is a group, $C, K \subset G$, and $w \in \cA^{CK}$. Then
    \begin{equation*}
        \mathcal{C}(w) \subset \cR_{C,K}^{-1} \bigl( \cR_{C,K}(w) \bigr),
    \end{equation*}
    and therefore if
    \begin{equation*}
         \cR_{C,K}^{-1} \bigl( \cR_{C,K}(w) \bigr) \subset \mathcal{C}(w),
    \end{equation*}
    then $w$ is identifiable.
    \begin{proof}
        
        Let $v \in \cC(w)$, i.e., there exists $g \in G$ such that $gC = C$, and $\sigma^g(v) = w$. 
        We wish to show that $\cR_{C,K}(w) = \cR_{C,K}(v)$. To start, we define $\phi : C \to C$ by $\phi(c) = gc$. This is clearly a permutation of $C$, since $C = gC$. Then, letting $c \in C$ and $k \in K$, we have that
        \begin{equation*}
            \sigma^c(w)(k) = w(ck) = \sigma^g(v)(ck) = \sigma^{gc}(v)(k) = \sigma^{\phi(c)}(v)(k),
        \end{equation*}
        and therefore $\sigma^c(w)(K) = \sigma^{\phi(c)}(v)(K)$. Then by Lemma \ref{read_eq}, we conclude that $\cR_{C,K}(w) = \cR_{C,K}(v)$, as desired.
    \end{proof}
\end{lemma}

%%%%%%%%%%%%%%%%%%%%%%%%%%%%%%%%%%%%%%%%%%%%%%%%%%%%%%%%
\subsection{Probability vectors and relevant properties}
%%%%%%%%%%%%%%%%%%%%%%%%%%%%%%%%%%%%%%%%%%%%%%%%%%%%%%%%

Here we introduce a family of complexity measures for discrete probability distributions, which will be used throughout this work. For a natural number $m$, we let
\begin{equation*}
    \Delta^m = \biggl\{ p = (p_1,\dots,p_m) \in [0,1]^m :  \sum_{k=1}^m p_k = 1 \biggr\}.
\end{equation*}
We use the same notation for the set of probability vectors indexed by any alphabet $\mathcal{A}$ with $|\mathcal{A}|=m$.

\begin{defn}
    Let $p$ be a probability vector on the finite alphabet $\cA$. For any natural number $i$, we define the function $\pi_i: \Delta^{|\cA|} \to (0,1]$ as
    \begin{equation*}
        \pi_i(p) = \sum_{a \in \cA} p_a^i.
    \end{equation*}
\end{defn}
We note that the second order R\'enyi entropy can be written in terms of $\pi_2(p)$: 
\begin{equation*}
    H_2(p) = -\ln\bigl(\pi_2(p)\bigr).
\end{equation*}

Next we require bounds on $\pi_i(p)$ in terms of $\pi_2(p)$. We use the standard $\ell^p$ norm on $\R^m$, which we denote by $\lVert \cdot \rVert_p$, and which satisfies the elementary property that if $1 \le p \le q$ and $x \in \R^m$, then $\lVert x \rVert_q \le \lVert x \rVert_p$. We include proofs of the following two lemmas for completeness.

\begin{lemma} \label{pvec:pi_upper}
    Let $p$ be a probability vector, and let $i \ge 2$ be a natural number. Then
    \begin{equation*}
        \pi_i(p) \le \pi_2(p)^{\frac{i}{2}}.
    \end{equation*}
    \begin{proof}
        By the elementary property mentioned above, we have that 
        $\lVert p \rVert_i \le \lVert p \rVert_2$ for any $i \ge 2$, and then
        \begin{equation*}
            \pi_i(p) = \left(\pi_i(p)^{\frac{1}{i}}\right)^i = \left(\lVert p \rVert_i\right)^i \le \left( \lVert p \rVert_2\right)^i = \left(\pi_2(p)^{\frac{1}{2}}\right)^i = \pi_2(p)^{\frac{i}{2}}.
        \end{equation*}
    \end{proof}
\end{lemma}

\begin{lemma} \label{pvec:pi_lower}
    Let $\cA$ be a finite set, and let $p = (p_a)_{a \in \cA}$ be a probability vector on $\cA$. Then for any natural number $i \ge 2$,
    \begin{equation*}
        \pi_i(p) \ge \pi_2(p)^{i-1}.
    \end{equation*}
    \begin{proof}
        Suppose $X$ is a real-valued random variable that takes the value $p_a$ with probability $p_a$, i.e., $\P(X = p_a) = p_a$. Let $\phi:[0,1] \to [0,1]$ be defined as $\phi(x) = x^{i-1}$. Since $i \ge 2$, $\phi$ is convex and we can apply Jensen's Inequality to get that
        \begin{equation*}
            \phi(\E[X]) \le \E[\phi(X)].
        \end{equation*}
        We now derive an expression for each side of this inequality. For the left hand side,
        \begin{equation*}
            \phi(\E[X]) = \phi\left(\sum_{a \in \cA} p_a \cdot p_a\right) = \phi\bigl(\pi_2(p)\bigr) = \pi_2(p)^{i-1}.
        \end{equation*}
        For the right hand side, we have
        \begin{equation*}
            \E[\phi(X)] = \sum_{a \in \cA} \phi(p_a) \cdot p_a = \sum_{a \in \cA} p_a^i = \pi_i(p).
        \end{equation*}
        Combining the previous three displays yields the desired result.
    \end{proof}
\end{lemma}

%%%%%%%%%%%%%%%%%%%%%%%%%%%%%%%%%%%%%%%%%%%%%%%%%%%%%%%%%%%%%%%%%%%%%%%%%%%
%%%%%%%%%%%%%%%%%%%%%%%%%%%%%%%%%%%%%%%%%%%%%%%%%%%%%%%%%%%%%%%%%%%%%%%%%%%
\section{Probabilistic framework and repeats in patterns} \label{sect:prob}
%%%%%%%%%%%%%%%%%%%%%%%%%%%%%%%%%%%%%%%%%%%%%%%%%%%%%%%%%%%%%%%%%%%%%%%%%%%
%%%%%%%%%%%%%%%%%%%%%%%%%%%%%%%%%%%%%%%%%%%%%%%%%%%%%%%%%%%%%%%%%%%%%%%%%%%

In this section we develop some basic probabilistic tools that will be used in later sections. Throughout this section, we suppose that we have a fixed shotgun identification problem on a group, where $\cA$, $p$, $G$, $C$, and $K$ are as in Section \ref{sect:background}. Here and throughout the rest of this work, we define $\mu_p$ to be the product measure $p^G$, which extends the definition of $\mu_p$ from $CK$ (as it was defined in Section \ref{Sect:Intro}) to all of $G$. We remind the reader that the main event of interest in this work is $\cI_{C,K}$, which is referred to as \textit{identifiability} and is formally defined as
\begin{equation*}
    \cI_{C,K} = \Bigl\{ x \in \cA^{G} : x(CK) \text{ is $(C,K)$-identifiable} \Bigr\} = \bigcup_{\substack{w \in \cA^{CK} \\ w \text{ is $(C,K)$-identifiable} } } [w].
\end{equation*}

The main purpose of this section is to prove Theorem \ref{prob:repeat_bounds}, which gives bounds on the probability that a randomly drawn pattern contains a repeated subpattern, and Theorem \ref{prob:n_disj_repeat_prob}, which gives the probability that a randomly drawn pattern contains $n$ disjoint repeated subpatterns. The notion of a \textit{repeat} is formally defined later in this section. We remark that many of the intermediate results in this section are true in greater generality than is stated here, but we do not pursue these generalizations as they are not used later in this work.

Let us begin by recording two elementary facts regarding product measures for future reference. First, suppose $\cA$ is a finite set and $p$ is a probability vector indexed by $\cA$. Further suppose that $G$ is a countable set and $\mu_p = p^G$ is the product measure on $\cA^G$. Then for any finite shape $F \subset G$ and any pattern $w \in \cA^F$, we have that
\begin{equation} \label{prob:fact1}
    \mu_p([w]) = \prod_{a \in \cA} (p_a)^{|\{ v \in F : w(v) = a\}|}.
\end{equation}
Next, for any finite set $F \subset G$, we let $\cS(F)$ denote the $\sigma$-algebra generated by all cylinder sets on $F$, i.e., $\cS(F)$ consists of all events that can be written as a union of some cylinder sets defined by patterns with shape $F$. Now suppose that $A_1,\dots,A_n$ are finite subsets of $G$ and for each $k = 1, \dots, n$, we have events $E_k \in \cS(A_k)$. If the sets $A_1,\dots,A_n$ are pairwise disjoint, then the collection $\{E_k\}_{k=1}^n$ is mutually independent, and in particular
\begin{equation} \label{prob:fact2}
    \mu_p \left( \bigcap_{k=1}^n E_k \right) = \prod_{k=1}^n \mu_p(E_k).
\end{equation}

%%%%%%%%%%%%%%%%%%%%%%%%%%%%%%
\subsection{Repeated patterns}
%%%%%%%%%%%%%%%%%%%%%%%%%%%%%%

In this section we turn our attention to estimating the probability of repeated patterns. First we give a precise definition of a repeat.
\begin{defn}
    Let $A \subset G$ be finite, and let $g \in G$. We say that $w \in \cA^{A \cup gA}$ has an $(A;g)$-\textit{repeat} if for each $h \in A$ we have $w(h) = w(gh)$. Furthermore, we let $\cR(A;g)$ denote the event that a randomly drawn configuration has an $(A;g)$-repeat: %\revision{}{(added '-' to $(A;g)$-repeat in math below)}
    \begin{equation*}
        \cR(A;g) = \bigcup_{\substack{w \in \cA^{A \cup gA} \\ w \text{ has an $(A;g)$-repeat}}} [w].
    \end{equation*}
\end{defn}

Now we define some additional notation and terminology that is helpful in analyzing repeats. 
In particular, we will define an equivalence relation on $A \cup gA$ such that if a pattern $w$ has an $(A;g)$-repeat and $h_0$ is equivalent to $h_1$, then $w(h_0) = w(h_1)$.
Let $A \subset G$ be finite, $g \in G$, and $F = A \cup gA$. Define the relation $\prec_g$ on $F$ by setting $u \prec_g v$ if $u \in A$ and $v = gu$. Note the following two facts:
\begin{itemize}
    \item[(i)]  if $u_1 \prec_g v$ and $u_2 \prec_g v$, then $g u_1  = v = g u_2$, and thus $u_1 = u_2$;
    \item[(ii)] if $u \prec_g v_1$ and $u \prec_g v_2$, then $v_1  = gu = v_2$, and thus $v_1 = v_2$.
\end{itemize} 
Now, let $\sim_g$ be the symmetric, transitive closure of $\prec_g$, which is then an equivalence relation on $F$. Consider any $u \sim_g v \in F$. If $u = v$, then $v = g^0u$. Otherwise, by the definition of $\sim_g$ as the symmetric, transitive closure of $\prec_g$, together with the facts (i) and (ii), we see that (swapping $u$ and $v$ if necessary) for some $n \in \N$, there must be $w_1, w_2, \dots, w_{n-1} \in A$ such that
\begin{equation*}
    u \prec_g w_1 \prec_g w_2 \prec_g \cdots \prec_g w_{n-1} \prec_g v.
\end{equation*}
Note this also covers the case that $u \prec_g v$. By definition, this means that $v = g^nu$, and $\forall k \in [1,n)$, we have $w_k = g^ku \in A$.

As $\sim_g$ is an equivalence relation, it describes a partition of $F$ into equivalence classes. We let $\O(A;g)$ denote this set of equivalence classes and call its elements \textit{orbits}. This relation plays a role in our analysis of repeated patterns via the following lemma.
\begin{lemma} \label{prob:eqrel}
    Suppose $A \subset G$ is finite, $g \in G$, and $F = A \cup g A$. For each $x \in \cR(A;g)$ and $u,v \in F$, if $u \sim_g v$, then $x(u) = x(v)$. 
    \begin{proof}
       Let $x \in \cR(A;g)$. Suppose $u \sim_g v$. First, if $u = v$, then clearly $x(u) = x(v)$. Now consider the case $u \ne v$. By interchanging the role of $u$ and $v$ if necessary, we have that $v = g^n u$ for some $n \ge 1$ and $g^k u \in A$ for all $0 \le k \le n-1$. Since $x \in \cR(A;g)$, for each $k = 1, \dots, n$, we have $x(g^k u) = x(g g^{k-1} u) = x(g^{k-1} u)$. Chaining these equalities together, we get $x(v) = x( g^n u) = x(g^{n-1} u) = \dots = x(u)$, which concludes the proof. 
    \end{proof}
\end{lemma}

The next lemma gives an expression for the probability of the event $\cR(A;g)$ in terms of these equivalence classes. 
\begin{lemma} \label{prob:repeat_prob}
    Suppose $A \subset G$ is finite and $g \in G$. Then
    \begin{equation*}
        \mu_p \bigl( \cR(A;g) \bigr) = \prod_{O \in \O(A;g)} \pi_{|O|}(p). 
    \end{equation*}
    \begin{proof}
       For an equivalence class $O \in \O(A;g)$, let $E(O,a)$ denote the event that every element of $O$ is labeled by the symbol $a$, i.e., $E(O,a) = \{ x \in \cA^G : \forall u \in O, \, x(u) = a\}$. First, observe that for each equivalence class $O$, we have that
       \begin{equation*}
           \mu_p \bigl( E(O,a) \bigr) = (p_a)^{|O|}, 
       \end{equation*}
       where we have used \eqref{prob:fact1}. Next, let $E(O) = \bigcup_{a \in \cA} E(O,a)$, %\revision{}{(inline "cup" to "bigcup")}, 
       and note that this is a disjoint union. Hence we have that
       \begin{equation} \label{prob:repeat_prob:eq1}
           \mu_p \bigl( E(O) \bigr) = \sum_{a \in \cA} \mu_p \bigl( E(O,a) \bigr) = \sum_{a \in \cA} (p_a)^{|O|} = \pi_{|O|}(p).
       \end{equation}
       Now observe that $\O(A;g)$ gives a partition of $F = A \cup gA$ into equivalence classes. In particular, these equivalence classes are all pairwise disjoint, and then by \eqref{prob:fact2} and the previous display, we have that
       \begin{equation*}
            \mu_p \left( \bigcap_{O \in \O(A;g)} E(O) \right)  = \prod_{O \in \O(A;g)} \mu_p \bigl( E(O) \bigr) = \prod_{O \in \O(A;g)} \pi_{|O|}(p).
       \end{equation*}
       Let us now check that 
       \begin{equation*}
           \cR(A;g) = \bigcap_{O \in \O(A;g)} E(O).
       \end{equation*}
       First, if $x \in \cR(A;g)$, then $x \in E(O)$ for each $O \in \O(A;g)$ by Lemma \ref{prob:eqrel}. Hence $\cR(A;g)$ is contained in the intersection. For the reverse inclusion, if $x \in E(O)$ for each $O \in \O(A;g)$, then we clearly have that $x(h) = x(gh)$ for all $h \in A$, which implies that $x \in \cR(A;g)$. 
       Finally, combining the two previous displays completes the proof of the lemma.
    \end{proof}
\end{lemma}

The following lemma is used below to give bounds on the probability of $\cR(A;g)$ that do not depend on the particular orbit structure $\O(A;g)$.

\begin{lemma} \label{prob:sum_bounds}
    Suppose $A \subset G$ and $g \in G$ with $g \neq e$. Then
    \begin{equation*}
       |A| \leq \sum_{O \in \O(A;g)} |O|
       \quad \text{and} \quad
       \sum_{O \in \O(A;g)} (|O|-1) \le |A|.
    \end{equation*}
    \begin{proof}
        Let $F = A \cup gA$. Since $\O(A;g)$ is a partition of $F$ and $A \subset F$, we clearly have 
        \begin{equation*}
            \sum_{O \in \O(A;g)} |O| = |F| \ge |A|,
        \end{equation*}
        which proves the first inequality in the statement of the lemma.
       
        For the second inequality, note that 
        \begin{equation} \label{prob:sum_bounds:eq1}
            \sum_{O \in \O(A;g)} (|O|-1) = \left(\sum_{O \in \O(A;g)} |O| \right) - |\O(A;g)| = |F| - |\O(A;g)|.
        \end{equation}
       
        Let us now show that if $u,v \in F \setminus g A$ and $u \ne v$, then $u$ is not equivalent to $v$. Suppose for contradiction that $u,v \in F \setminus g A$, $u \ne v$, and $u \sim_g v$. By interchanging the roles of $u$ and $v$ if necessary, we may assume without loss of generality that $v = g^nu$ for some $n \geq 0$ and $g^k u \in A$ for all $0 \le k \le n-1$. Since $u \ne v$, we must have $g^n \ne e$, and therefore $n \ge 1$. Then $g^{n-1} u \in A$, and we have $v = g^n u= g (g^{n-1} u) \in gA$, which contradicts the assumption that $v \notin gA$.
       
        By the previous paragraph, we see that if $u,v \in F \setminus gA$ and $u \ne v$, then $u$ is not equivalent to $v$. As a result, the map from $F \setminus g A$ to $\O(A;g)$ that sends $u$ to the orbit containing $u$ is an injection, and therefore $|F \setminus gA| \le |\O(A;g)|$. Combining this inequality with \eqref{prob:sum_bounds:eq1} gives
        \begin{equation*}
            \sum_{O \in \O(A;g)} (|O|-1) = |F| - |\O(A;g)| \le |F| - |F \setminus gA| = |F \setminus (F \setminus gA)| = |gA| = |A|,
        \end{equation*}
        which completes the proof of the lemma. 
    \end{proof}
\end{lemma}

Now we give bounds on $\mu_p \bigl( \cR(A;g) \bigr)$ that depend only on $|A|$ and $\pi_2(p)$. 

\begin{thm} \label{prob:repeat_bounds}
    Let $G$ be a countable group, and let $p$ be a probability vector. Then for any finite $A \subset G$ and any $g \in G$,
    \begin{equation*}
     \pi_2(p)^{|A|} \le \mu_p \bigl( \cR(A;g) \bigr) \le \pi_2(p)^{\frac{1}{2} |A|}.
    \end{equation*}
    \begin{proof}
        Let $G$, $p$, $A$, and $g$ be as above. By Lemma \ref{prob:repeat_prob}, we have
        \begin{equation} \label{prob:repeat_bounds:eq1}
            \mu_p \bigl( \cR(A;g) \bigr) = \prod_{O \in \O(A;g)} \pi_{|O|}(p).
        \end{equation}
        Let us first establish the upper bound. By \eqref{prob:repeat_bounds:eq1} and Lemmas \ref{pvec:pi_upper} and \ref{prob:sum_bounds}, we have
        \begin{align*}
            \mu_p \bigl( \cR(A;g) \bigr)  = \prod_{O \in \O(A;g)} \pi_{|O|}(p) 
            & \le \prod_{O \in \O(A;g)} \pi_{2}(p)^{\frac{1}{2}|O|} \\
            & = \pi_{2}(p)^{\frac{1}{2} \sum_{O \in \O(A;g)} |O|} 
             \leq \pi_{2}(p)^{\frac{1}{2} |A|}. 
        \end{align*}
        Now we establish the lower bound. By \eqref{prob:repeat_bounds:eq1} and Lemmas \ref{pvec:pi_lower} and \ref{prob:sum_bounds}, we have
        \begin{align*}
            \mu_p \bigl( \cR(A;g) \bigr)  = \prod_{O \in \O(A;g)} \pi_{|O|}(p) 
            & \geq \prod_{O \in \O(A;g)} \pi_{2}(p)^{|O|-1} \\
            & = \pi_2(p)^{ \sum_{O \in \O(A;g)} (|O|-1) } 
             \geq \pi_2(p)^{|A|}.
       \end{align*}
       This concludes the proof of the theorem. 
    \end{proof}
\end{thm}

In our proofs in Section \ref{sect:neg}, we also require estimates regarding the probability that a single pattern is repeated multiple times on disjoint shapes. Toward that end, we extend our original notation. Let $A \subset G$ be a finite set, let $g_1,\dots,g_n \in G$ with $g_1 = e$, and let $F = \cup_{k=1}^n g_k A$. A pattern $w \in \cA^F$ has an $(A; g_2,\dots,g_n)$-repeat if for each $h \in A$, there is a symbol $a_h$ such that $w(g_k h) = a_h$ for all $k = 1,\dots,n$. Now let $\cR(A; g_2,\dots,g_n)$ denote the event that $x \in \cA^G$ has an $(A; g_2,\dots,g_n)$-repeat, i.e.,
\begin{equation*}
    \cR(A; g_2,\dots,g_n) = \bigcup_{ \substack{ w \in \cA^F \\ w \text{ has an $(A; g_2,\dots,g_n)$-repeat}}} [w].
\end{equation*}

We remark that it is possible to give an exact expression for the probability of this event regardless of the choice of $g_2, \dots, g_n$; however, we do not require such a general result. We only require an exact probability when the collection $\{ g_k A \}_{k=1}^n$ is mutually disjoint, which is far easier to establish. The following theorem gives our second main result for the section. 

\begin{thm} \label{prob:n_disj_repeat_prob}
    Suppose $A \subset G$ is finite, $g_1,\dots, g_n \in G$ with $g_1 = e$, and the collection $\{ g_k A \}_{k=1}^n$ is mutually disjoint. Then
    \begin{equation*}
        \mu_{p} \bigl( \cR(A; g_2,\dots,g_n) \bigr) = \pi_n(p)^{|A|}.
    \end{equation*}
    \begin{proof}
        Let $A$ and $g_1,\dots,g_n$ be as in the statement of the theorem. Let $F = \cup_{k=1}^n g_k A$. For each $h \in A$, let $O(h) = \{ g_k h : 1 \leq k \leq n\}$, and let $\O(A; g_2,\dots,g_n) = \{ O(h) : h \in A\}$. Since $\{g_k A\}_{k=1}$ is mutually disjoint, we have that $|O(h)| = n$ for each $h \in A$, and furthermore $\O(A; g_2,\dots,g_n)$ is a partition of $F$. For $O \in \O(A; g_2,\dots,g_n)$ and $a \in \cA$, let $E(O,a) = \{ x \in \cA^G : \forall h \in O, \, x(h) = a\}$. Next let $E(O) = \cup_{a \in \cA} E(O,a)$. Then
        \begin{equation*}
            \mu_{p} \bigl( E(O,a) \bigr) = (p_a)^{|O|} = (p_a)^n,
        \end{equation*}
        and therefore
        \begin{equation*}
            \mu_p \bigl( E(O) \bigr) = \sum_{a \in \cA} \mu_p \bigl( E(O,a) \bigr) = \sum_{a \in \cA} (p_a)^n = \pi_n(p). 
        \end{equation*}
        Since each map $h \mapsto g_k h$ is a bijection from $A$ to $g_kA$, we also have that the map $h \mapsto O(h)$ is a bijection from $A$ to $\O(A; g_2,\dots, g_n)$. Hence $|\O(A; g_2, \dots, g_n)| = |A|$. Additionally, note that
        \begin{equation*}
            \cR(A; g_2,\dots, g_n) = \bigcap_{O \in \O(A; g_2,\dots,g_n) }E(O).
        \end{equation*}
        Combining all of these facts, we see that
        \begin{align*}
            \mu_p \bigl( \cR(A; g_2,\dots,g_n) \bigr) & = \mu_p \left( \bigcap_{O \in \O(A; g_2,\dots,g_n)} E(O) \right) \\
            & = \prod_{O \in \O(A; g_2,\dots,g_n)} \mu_p \bigl( E(O) \bigr) \\
            & = \prod_{O \in \O(A; g_2,\dots,g_n)} \pi_n(p) \\
            & = \pi_n(p)^{|\O(A;g_2,\dots,g_n)|} \\
            & = \pi_n(p)^{|A|},
        \end{align*}
        as desired.
    \end{proof}
\end{thm}

We note that Theorems \ref{prob:repeat_bounds} and \ref{prob:n_disj_repeat_prob} exactly recover Lemma 11 in \cite{motahari} for the case of DNA sequences, and similar results appear in \cite{mossel} in the case of $d$-dimensional integer lattices. 

%%%%%%%%%%%%%%%%%%%%%%%%%%%%%%%%%%%%%%%%%%%%%%%%%%%%%%%
%%%%%%%%%%%%%%%%%%%%%%%%%%%%%%%%%%%%%%%%%%%%%%%%%%%%%%%
\section{Identifiability results} \label{sect:pos}
%%%%%%%%%%%%%%%%%%%%%%%%%%%%%%%%%%%%%%%%%%%%%%%%%%%%%%%
%%%%%%%%%%%%%%%%%%%%%%%%%%%%%%%%%%%%%%%%%%%%%%%%%%%%%%%

In this section we aim to prove Theorem \ref{pos:regime}, which establishes sufficient conditions for identifiability a.a.s. The general idea is that if the reads can be organized into a connected graph with sufficiently large and regular overlapping regions between reads connected by an edge, then identifiability should occur with high probability.
In fact, we show that under the hypotheses of the theorem, the overlapping regions all have unique labels with high probability and this unique labeling condition implies identifiability. In the next section, we establish some combinatorial lemmas, including Lemma \ref{pos:id_lem}, which shows that the unique labeling condition implies identifiability. In Section \ref{Sect:ProbIdBelow} we find a lower bound on the probability of identifiability by estimating the probability of satisfying the unique labeling condition. Finally, we use these results to prove Theorem \ref{pos:regime} in Section \ref{Sect:ProofPositive}.

%%%%%%%%%%%%%%%%%%%%%%%%%%%%%%%%%%%%%%%%%%%%%%%%%%%%%%%%%%%%%%%%%%%%%%%%%%%%%
\subsection{Combinatorial results for identifiability} \label{Sect:CombIdent}
%%%%%%%%%%%%%%%%%%%%%%%%%%%%%%%%%%%%%%%%%%%%%%%%%%%%%%%%%%%%%%%%%%%%%%%%%%%%%

To begin our analysis, we  recall the definition of an overlap graph.
\begin{defn}
    Let $G$ be a group, and let $C,K \subset G$ and $\bF \subset \cP(K)$. The \textit{overlap graph} $\cO = \cO(G,C,K,\bF)$ is the undirected graph with vertex set $V(\cO) = C$ and edge set
    \begin{equation*}
        E(\cO) = \Bigl\{(c_1, c_2) \in C^2 : \exists g \in G, \exists F \in \bF, gF \subset c_1K \cap c_2K \Bigr\}.
    \end{equation*}
\end{defn}

The following lemma formalizes an intuitive property of overlap graphs: weakening the requirements for including edges cannot disconnect an overlap graph. We use it in the proof of Theorem \ref{pos:regime}.

\begin{lemma} \label{pos:sub-overlap_connected}
    Let $G$ be a group, let $C, K \subset G$, %\revision{and suppose that}{let} 
    let $\bF \subset \cP(K)$ and suppose that %\revision{}{suppose that} $\cO(G,C,K,\bF)$
    is connected. If $\bF' \subset \cP(K)$ and $\forall F \in \bF$, $\exists F' \in \bF'$ such that $F' \subset F$, then $\cO(G,C,K,\bF')$ is connected.
    \begin{proof}
        Let $\cO' = \cO(G,C,K,\bF')$, and consider some edge $(c_1, c_2) \in E(\cO)$. Then $\exists g \in G$ and $\exists F \in \bF$ such that $gF \subset c_1K \cap c_2K$. We also have by assumption that $\exists F' \in \bF'$ such that $F' \subset F$, and therefore $gF' \subset gF$ and $gF' \subset c_1K \cap c_2K$. Hence $(c_1, c_2) \in E(\cO')$. Therefore, $\cO$ is a subgraph of $\cO'$ containing all the vertices. Since $\cO$ is connected, $\cO'$ is connected.
    \end{proof}
\end{lemma}

The next lemma provides a sufficient condition for identifiability. Informally, it states that if a pattern has unique labels on all overlap regions, then the pattern is identifiable. We note that an analogous lemma appears as \cite[Lemma 2.4]{mossel} in the setting of graphs.

\begin{lemma} \label{pos:id_lem}
    Let $G$ be a group, let $C, K \subset G$, and suppose there exists $\bF \subset \cP(K)$ such that $\varnothing \notin \bF$ and $\cO(G,C,K,\bF)$ is connected. Suppose that $x \in \cA^G$ has unique patterns on all translates of each $F \in \bF$ contained in $CK$, i.e., $\forall g,h \in G$, $\forall F \in \bF$, if $gF, hF \subset CK$ and $\sigma^g(x)(F) = \sigma^h(x)(F)$, then $g = h$. Then $x(CK)$ is $(C,K)$-identifiable, i.e., $x \in \cI_{C,K}$.
    \begin{proof}
        For simplicity, let $\cO = \cO(G,C,K,\bF)$ and let $w = x(CK)$, and then consider some $w' \in \cA^{CK}$ such that $\cR_{C,K}(w) = \cR_{C,K}(w')$. By Lemma \ref{read_eq}, there exists a permutation $\phi$ of $C$ such that $\forall c \in C, k \in K$,
        \begin{equation*}
            \sigma^c(w)(k) = \sigma^{\phi(c)}(w')(k).
        \end{equation*}
        Take any $c_0 \in C$, and let $g_1 = \phi(c_0)c_0^{-1}$. Now define 
        \begin{equation*}
            T = \bigl\{c \in C : \phi(c) = g_1c\bigr\}.
        \end{equation*}
        Note that $c_0 \in T$, since $\phi(c_0) = \phi(c_0)c_0^{-1}c_0 = g_1c_0$, and thus $T$ is non-empty. Also $T \subset C$. We wish to show that $\phi(c) = g_1c$ for all $c \in C$, or in other words, $T = C$. First we establish the following statement.
        \begin{claim*}
            Let $t \in T$ and $c \in C$, and suppose that %\revision{consider some edge}{suppose that}
            $(\phi(t), \phi(c)) \in E(\cO)$. Then $c \in T$.
            
            We now prove the claim.  Let $t \in T$ and $c \in C$, and suppose that $(\phi(t),\phi(c)) \in E(\cO)$. By definition, we have that $\phi(t) = g_1t$, and letting $g_2 = \phi(c)c^{-1}g_1^{-1}$, we have $\phi(c) = g_2g_1c$. We aim now to show that $g_2 = e$. Since $(\phi(t),\phi(c)) \in E(\cO)$, there exists $h \in G$ and $F \in \bF$ such that
            \begin{equation*}
                hF \subset \phi(t)K \cap \phi(c)K.
            \end{equation*}
            Hence $w'(hF)$ appears in both $w'(\phi(t)K)$ and $w'(\phi(c)K)$.
            
            First, we wish to show that the pattern $w'(hF)$ appears in the read $w(tK)$, which we do using the fact that $hF \subset \phi(t)K$. In particular, we will show that $g_1^{-1}hF \subset tK$, and $w(g_1^{-1}hF)$ is the same pattern as $w'(hF)$. Taking some $f \in F$, we have $hf \in hF \subset \phi(t)K$, so we can find some $k_1 \in K$ such that $hf = \phi(t)k_1$, which gives
            \begin{equation*}
                w'(hf) = w'(\phi(t)k_1) = \sigma^{\phi(t)}(w')(k_1).
            \end{equation*}
            Next, since $t \in C$ and $k_1 \in K$, we have that $\sigma^t(w)(k_1) = \sigma^{\phi(t)}(w')(k_1)$, and combining this fact with the previous display yields
            \begin{equation*}
                w'(hf) = \sigma^{\phi(t)}(w')(k_1) = \sigma^{t}(w)(k_1) = w(tk_1).
            \end{equation*}
            Since $tk_1 = g_1^{-1}g_1tk_1 = g_1^{-1}\phi(t)k_1 = g_1^{-1}hf$, we obtain that $w'(hf) = w(t k_1) = w(g_1^{-1}hf)$. Since $f \in F$ was arbitrary, we have shown that $\forall f \in F$, 
            \begin{equation*}
                w'(hf) = w(g_1^{-1}hf),
            \end{equation*}
            and thus $\sigma^h(w')(F) = \sigma^{g_1^{-1}h}(w)(F)$. 
            
            Now we wish to show that the pattern $w'(hF)$ appears in the pattern $w(cK)$, which we do using the fact that $hF \subset \phi(c)K$. In particular, we will show that $g_1^{-1}g_2^{-1}hF \subset cK$, and $w(g_1^{-1}g_2^{-1}hF)$ is the same pattern as $w'(hF)$. Taking some $f \in F$, we have $hf \in hF \subset \phi(c)K$, so we can find some $k_2 \in K$ such that $hf = \phi(c)k_2$, which gives
            \begin{equation*}
                w'(hf) = w'(\phi(c)k_2) = \sigma^{\phi(c)}(w')(k_2).
            \end{equation*}
            Next, since $c \in C$ and $k_2 \in K$, we have that $\sigma^c(w)(k_2) = \sigma^{\phi(c)}(w')(k_2)$, and combining this fact with the previous display gives
            \begin{equation*}
                w'(hf) = \sigma^{\phi(c)}(w')(k_2) = \sigma^{c}(w)(k_2) = w(ck_2).
            \end{equation*}
            Since $ck_2 = (g_2g_1)^{-1}g_2g_1ck_2 = g_1^{-1}g_2^{-1}\phi(c)k_2 =  g_1^{-1}g_2^{-1}hf$, we obtain that $w'(hf) = w(ck_2) = w(g_1^{-1} g_2^{-1} hf)$. Since $f \in F$ was arbitrary, we have shown that $\forall f \in F$,
            \begin{equation*}
                w'(hf) = w(g_1^{-1}g_2^{-1}hf),
            \end{equation*}
            and thus $\sigma^h(w')(F) = \sigma^{g_1^{-1}g_2^{-1}h}(w)(F)$. 
            
            By the conclusions of the two previous paragraphs, we have that
            \begin{equation} \label{eqn:computermen}
                \sigma^{g_1^{-1}h}(w)(F) = \sigma^h(w')(F) = \sigma^{g_1^{-1}g_2^{-1}h}(w)(F).
            \end{equation}
            Since $hF \subset \phi(t)K = g_1tK$, we have $g_1^{-1}hF \subset tK \subset CK$, and since $hF \subset \phi(c)K = g_2g_1cK$, we have $g_1^{-1}g_2^{-1}hF \subset cK \subset CK$. Then by the uniqueness property assumed in the lemma and \eqref{eqn:computermen}, we obtain that
            \begin{equation*}
                g_1^{-1}h = g_1^{-1}g_2^{-1}h.
            \end{equation*}
            After cancelling like terms, we see that $g_2 = e$. 
            %Informally, \eqref{eqn:computermen} seems to imply the same $F$-shaped pattern appears twice in $w$, however the uniqueness of $F$-shaped patterns in $w$ implies these must be the same pattern. 
            As a result, we have
            \begin{equation*}
                \phi(c) = g_2g_1c = g_1c,
            \end{equation*}
            and therefore $c \in T$. This concludes the proof of the claim.
        \end{claim*}
        
        Now consider any $c \in C$. Since $\cO$ is connected, there must be some path between $\phi(c_0)$ and $\phi(c)$, with edges $(\phi(c_0), \phi(c_1))$, $(\phi(c_1), \phi(c_2))$, \dots, $(\phi(c_n), \phi(c))$, where $c_1, \dots, c_n \in C$. Since $c_0 \in T$ and $(\phi(c_0),\phi(c_1)) \in E(\cO)$, the claim yields that $c_1 \in T$. Then, by applying the claim inductively, we conclude that $c_2, c_3, \dots, c_n \in T$, and finally that $c \in T$. Thus, $C \subset T$, and therefore we obtain that $T = C$ and $\phi(C) = g_1C$.
        
        Since $\phi$ is a permutation, we also know that $\phi(C) = C$, and so $C = g_1C$. Now let $h \in CK$, and choose $c \in C$ and $k \in K$ such that $h = ck$. By the definition of $\phi$ and the fact that $\phi(c) = g_1 c$ for all $c \in C$, we see that for any $c \in C$ and $k \in K$, we must have $\sigma^c(w)(k) = \sigma^{\phi(c)}(w')(k) = \sigma^{g_1c}(w')(k)$, and then
        \begin{equation*}
            w(h) = w(ck) = \sigma^c(w)(k) = \sigma^{g_1c}(w')(k) = \sigma^{g_1}(w')(ck) = \sigma^{g_1}(w')(h).
        \end{equation*}
        Since $h \in CK$ was arbitrary, we conclude that $\sigma^{g_1}(w') = w$, and thus
        \begin{equation*}
            w' \in \cC(w).
        \end{equation*}
        Since $w' \in \cR_{C,K}^{-1}\bigl(\cR_{C,K}(w)\bigr)$ %\revision{}{(make outer parens big in math below and before)} 
        was arbitrary, we have shown that
        \begin{equation*}
            \cR_{C,K}^{-1}\bigl(\cR_{C,K}(w)\bigr) \subset \cC(w).
        \end{equation*}
        Finally, by Lemma \ref{id_char}, we conclude that $w = x(CK)$ is identifiable, and thus $x \in \cI_{C,K}$.
    \end{proof}
\end{lemma}

The following elementary lemma provides a characterization of when two translates of a set overlap.

\begin{lemma} \label{pos:trans_overlap}
    Let $G$ be a group, let $g,h \in G$, and let $A \subset G$. Then
    \begin{equation*}
        gA \cap hA \ne \varnothing \iff g \in hAA^{-1}
    \end{equation*}
    \begin{proof}
        Assuming $gA \cap hA$ is nonempty, we have some $c \in gA$ and $c \in hA$. Then there are some $a_1, a_2 \in A$ such that $c = ga_1 = ha_2$, and therefore $g = ha_2a_1^{-1}$. Hence $g \in hAA^{-1}$.
        
        Assuming $g \in hAA^{-1}$, we have some $a_1, a_2 \in A$ such that $g = ha_2a_1^{-1}$. Thus, we have that $ga_1 = ha_2$. Letting $x = ga_1 = ha_2$, we clearly have that $x = ga_1 \in gA$ and $x = ha_2 \in hA$, and therefore $gA \cap hA \ne \varnothing$. 
    \end{proof}
\end{lemma}

%%%%%%%%%%%%%%%%%%%%%%%%%%%%%%%%%%%%%%%%%%%%%%%%%%%%%%%%%%%%%%%%%%%
\subsection{Bounding the probability of identifiability from below} \label{Sect:ProbIdBelow}
%%%%%%%%%%%%%%%%%%%%%%%%%%%%%%%%%%%%%%%%%%%%%%%%%%%%%%%%%%%%%%%%%%%

Now we bound the probability of identifiability from below using the sufficient condition given in Lemma \ref{pos:id_lem}. The basic idea is to estimate the probability of having unique patterns on all translates of the overlap shapes $F \in \bF$. Here we use the estimates on the probability of repeats derived in Section \ref{sect:prob}.

\begin{lemma} \label{pos:id_bound}
    Let $G$ be a group, let $C, K \subset G$, let $\cA$ be a finite alphabet, let $p$ be a probability vector on $\cA$, and let $\mu_p$ be the product measure $p^G$ on $\cA^G$. If there exists $\bF \subset \cP(K)$ such that $\varnothing \notin \bF$ and $\cO(G,C,K,\bF)$ is connected, then
    \begin{equation*}
        \mu_p(\cI_{C,K}) \ge 1 - \sum_{F \in \bF} \left[ |CK|^2\pi_2(p)^{|F|} + |CK||FF^{-1}|\pi_2(p)^{\frac{1}{2}|F|} \right].
    \end{equation*}
    \begin{proof}
        For any $F \in \bF$, define the set
        \begin{equation*}
            T(F) = \{g \in G : gF \subset CK\}.
        \end{equation*}
        Then, for any two-element subset $\{a, b\}$ of $T(F)$, we define an indicator random variable $\1^F_{a,b}$ on the probability space $(\cA^G,\mu_p)$ by the rule
        \begin{equation*}
            \1^F_{a,b}(x) = \begin{cases}
                1, & x \in \cR(aF; ba^{-1}) \\
                0, & \text{otherwise}.
            \end{cases}
        \end{equation*}
        Recall that the event $\cR(aF; ba^{-1})$ contains all patterns $x$ for which $\sigma^a(x)(F) = \sigma^b(x)(F)$. Now, letting $T^*(F)$ denote the set of two-element subsets of $T(F)$ (making $|T^*(F)| = \binom{|T(F)|}{2}$), we define the random variable $X$ by
        \begin{equation*} 
            X = \sum_{F \in \bF} \sum_{\{a,b\} \in T^*(F)} \1^F_{a,b}.
        \end{equation*}
        Note that $X$ counts the number of translates of any $F \in \bF$ in $CK$ that are labeled identically. In particular, if $X = 0$, then all translates of each $F \in \bF$ in $CK$ are uniquely labeled (as in the hypotheses of Lemma \ref{pos:id_lem}). Thus, the additional hypothesis that the overlap graph is connected gives that $(X = 0) \subset \cI_{C,K}$ by Lemma \ref{pos:id_lem}. Hence, by monotonicity, we have
        \begin{equation*} 
            \mu_p(\cI_{C,K}) \ge \mu_p(X = 0) = 1 - \mu_p(X \ne 0).
        \end{equation*}
        Since $X$ takes the value of a non-negative integer, the event that $X \ne 0$ is the same as the event that $X \ge 1$. Then, applying Markov's Inequality, we get
        \begin{equation} \label{eqn:raven}
            \mu_p(\cI_{C,K}) \ge 1 - \mu_p(X \ge 1) \ge 1 - \E[X].
        \end{equation}
        Let us now focus on bounding $\E[X]$ from above. We start by noting that
        \begin{equation*}
            \E[X] = \sum_{F \in \bF} \sum_{\{a,b\} \in T^*(F)} \E\left[\1^F_{a,b}\right] = \sum_{F \in \bF} \sum_{\{a,b\} \in T^*(F)} \mu_p\bigl(\cR(aF; ba^{-1})\bigr).
        \end{equation*}
        
        Next we partition $T^*(F)$ into $T^*_D(F)$ and $T^*_I(F)$, the disjoint and intersecting sets respectively, defined as
        \begin{equation*}
            T^*_D(F) = \Bigl\{\{a,b\} \in T^*(F) : aF \cap bF = \varnothing \Bigr\} \quad \text{and} \quad T^*_I(F) = T^*(F) \setminus T^*_D(F).
        \end{equation*}
        %and $T^*_I(F) = T^*(F) \setminus T^*_D(F)$. 
        
        With this partition of $T^*(F)$, we can rewrite $\E[X]$ as
        \begin{align*} 
            \E[X] &= \sum_{F \in \bF} \sum_{\{a,b\} \in T^*(F)} \mu_p\bigl(\cR(aF;ba^{-1})\bigr)\\
            &= \sum_{F \in \bF} \left[ \sum_{\{a,b\}  \in T^*_D(F)} \mu_p\bigl(\cR(aF; ba^{-1})\bigr) + \sum_{\{a,b\} \in T^*_I(F)} \mu_p\bigl(\cR(aF;ba^{-1})\bigr) \right].
        \end{align*}
        For any $\{a,b\} \in T^*_D(F)$, since $aF \cap bF = \varnothing$, Theorem \ref{prob:n_disj_repeat_prob} gives that
        \begin{equation*}
            \mu_p\bigl(\cR(aF;ba^{-1})\bigr) = \pi_2(p)^{|aF|} = \pi_2(p)^{|F|}.
        \end{equation*}
        For any $\{a,b\} \in T^*_I(F)$, Theorem \ref{prob:repeat_bounds} gives that
        \begin{equation*}
            \mu_p\bigl(\cR(aF;ba^{-1})\bigr) \le \pi_2(p)^{\frac{1}{2}|aF|} = \pi_2(p)^{\frac{1}{2}|F|}.
        \end{equation*}
        By combining the three previous displays and simplifying, we see that
        \begin{align} \begin{split} \label{eqn:walkman}
            \E[X] &= \sum_{F \in \bF} \left[ \sum_{\{a,b\} \in T^*_D(F)} \mu_p\bigl(\cR(aF;ba^{-1})\bigr) + \sum_{\{a,b\} \in T^*_I(F)} \mu_p\bigl(\cR(aF;ba^{-1})\bigr) \right] \\
            &\le \sum_{F \in \bF} \left[ \sum_{\{a,b\} \in T^*_D(F)} \pi_2(p)^{|F|} + \sum_{\{a,b\} \in T^*_I(F)} \pi_2(p)^{\frac{1}{2}|F|} \right] \\
            &= \sum_{F \in \bF} \left[ |T^*_D(F)|\pi_2(p)^{|F|} + |T^*_I(F)|\pi_2(p)^{\frac{1}{2}|F|} \right].
            \end{split}
        \end{align}
        
        Let us find an upper bound for $|T^*_D(F)|$. By definition of $T(F)$, we have that $T(F)F \subset CK$. Since each $F \in \bF$ is nonempty, we also have that $|T(F)| \le |T(F)F|$, and by definition, $T^*_D(F) \subset T^*(F)$. By combining these facts, we obtain
        \begin{equation*}
            |T^*_D(F)| \le |T^*(F)| = \binom{|T(F)|}{2} \le |T(F)|^2 \le |T(F)F|^2 \le |CK|^2.
        \end{equation*}
        
        Let us now find an upper found on $|T^*_I(F)|$. Let $a \in T(F)$, and consider all possible $b \in T(F)$ such that $aF \cap bF \ne \varnothing$. By Lemma \ref{pos:trans_overlap}, we have that $aF \cap bF \ne \varnothing \iff b \in aFF^{-1}$, and thus, for any choice of $a \in T(F)$, the only choices for $b$ are in $aFF^{-1}$. As a result, there are at most $|aFF^{-1}| = |FF^{-1}|$ choices for $b$ given any choice of $a$. Furthermore, there are clearly at most $|T(F)|$ choices for $a$, which we can again bound by $|T(F)| \le |T(F)F| \le |CK|$. Thus, with at most $|CK|$ choices for $a$ and at most $|FF^{-1}|$ corresponding choices for $b$, we see that $|T^*_I(F)| \le |CK||FF^{-1}|$.
        
        Combining our bounds on $|T^*_D(F)|$ and $|T^*_I(F)|$ with \eqref{eqn:walkman}, we get
        \begin{align*}
            \E[X] &\le \sum_{F \in \bF} |T^*_D(F)|\pi_2(p)^{|F|} + |T^*_I(F)|\pi_2(p)^{\frac{1}{2}|F|} \\
            &\le \sum_{F \in \bF} |CK|^2\pi_2(p)^{|F|} + |CK||FF^{-1}|\pi_2(p)^{\frac{1}{2}|F|}.
        \end{align*}
        Then by \eqref{eqn:raven} and the previous display, we obtain
        \begin{equation*}
            \mu_p(\cI_{C,K}) \ge 1 - \E[X] \ge 1 - \sum_{F \in \bF} \left[ |CK|^2\pi_2(p)^{|F|} + |CK||FF^{-1}|\pi_2(p)^{\frac{1}{2}|F|} \right],
        \end{equation*}
        which completes the proof.
    \end{proof}
\end{lemma}

%%%%%%%%%%%%%%%%%%%%%%%%%%%%%%%%%%%%%%%%%%%%%%
\subsection{Proof of Theorem \ref{pos:regime}} \label{Sect:ProofPositive}
%%%%%%%%%%%%%%%%%%%%%%%%%%%%%%%%%%%%%%%%%%%%%%

We are now prepared to prove Theorem \ref{pos:regime}, which we restate for the reader's convenience. 

\begin{repthm}{pos:regime}
    Let $\cA$, $p$, $\{G_n\}$, $\{C_n\}$, $\{K_n\}$, and $\{\mu_p^n\}$ be as in Section \ref{Sect:Intro}, and suppose that $\lim_n |C_nK_n| = \infty$ \eqref{assump:CK_to_inf}. If there exists a sequence $\{\bF_n\}$ such that $\bF_n \subset \cP(K_n)$ and such that \begin{itemize}[itemsep=0.4\baselineskip]
        \item[\eqref{assump:overlap_conn}] $\forall n \gg 0$, $\cO(G_n,C_n,K_n,\bF_n)$ is connected,
        \item[\eqref{assump:few_oshapes}] $\lim_n \frac{\ln |\bF_n|}{\ln |C_nK_n|} = 0$, and
        \item[\eqref{assump:oshape_size_bound}] $\exists \epsilon > 0$ s.t. $\forall n \gg 0$, $\forall F\in \bF_n$ we have $|F| \geq (1+\epsilon)\frac 2{H_2(p)} \ln |C_nK_n|$,
    \end{itemize}
    then identifiability holds a.a.s.
    \begin{proof}
        First, consider $N$ large enough that $|C_n K_n| >1$ for all $n \ge N$, and then let $\alpha$ be some constant such that for every index $n \ge N$,
        \begin{equation*}
            \alpha > \frac{1}{\ln(|C_nK_n|)} + (1 + \epsilon)\frac{2}{H_2(p)}.
        \end{equation*}
        
        With this definition of $\alpha$, we also have that
        \begin{equation*}
            \alpha\ln(|C_nK_n|) > 1 + (1 + \epsilon)\frac{2}{H_2(p)}\ln(|C_nK_n|),
        \end{equation*}
        and thus, for any index $n \ge N$, we can find some integer $\beta_n$ such that 
        \begin{equation*}
            (1 + \epsilon)\frac{2}{H_2(p)}\ln(|C_nK_n|) < \beta_n \le \alpha\ln(|C_nK_n|).
        \end{equation*}
        Now, for each $n \ge N$, we define another set of overlap shapes $\bar{F}_n$ as follows. For any $n \ge N$, let $F \in \bF_n$. If
        \begin{equation*}
            |F| \le \alpha\ln(|C_nK_n|),
        \end{equation*}
        then we include $F$ in $\bar{\bF}_n$, and we note that $|F| \geq (1+\epsilon) 2 \ln(|C_n K_n|) / H_2(p)$ by \eqref{assump:oshape_size_bound}. Otherwise, arbitrarily choose a subset $\bar{F} \subset F$ containing exactly $\beta_n$ elements, and include $\bar{F}$ in $\bar{\bF}_n$. Note that $\forall n \ge N, \forall \bar{F} \in \bar{\bF}_n$, we have that
        \begin{equation} \label{eqn:coffee}
            (1+\epsilon)\frac{2}{H_2(p)}\ln(|C_nK_n|) \le |\bar{F}| \le \alpha\ln(|C_nK_n|).
        \end{equation}
        Furthermore, $\forall \bar{F} \in \bar{\bF}_n$, there is some $F \in \bF_n$ such that $\bar{F} \subset F \subset G_n$. Since $|C_nK_n| > 1$, we must have that $\beta_n > 0$, so we know that each $\bar{F}$ is nonempty. Therefore each $\bar{\cO}_n = \cO(G_n,C_n,K_n,\bar{\bF}_n)$ is connected by Lemma \ref{pos:sub-overlap_connected}. 
        
        Now consider a single index $n \ge N$. Since each $F \in \bar{\bF}_n$ is nonempty and each $\bar{\cO}_n$ is connected, Lemma \ref{pos:id_bound} gives that
        \begin{equation} \label{eqn:sugar}
            \mu^n_p(\cI_{C_n,K_n}) \ge 1 - \sum_{\bar{F} \in \bar{\bF}_n} \left[ |C_nK_n|^2\pi_2(p)^{|\bar{F}|} + |C_nK_n||\bar{F}\bar{F}^{-1}|\pi_2(p)^{\frac{1}{2}|\bar{F}|} \right].
        \end{equation}
        We now consider a single term in the sum. Using the lower bound on  $|\bar{F}|$ from \eqref{eqn:coffee}, we obtain that
        \begin{align*}
            |C_nK_n|^2\pi_2(p)^{|\bar{F}|} &= \exp\bigl(2\ln(|C_nK_n|)\bigr)\exp\bigl(-|\bar{F}|H_2(p)\bigr) \\
            &\le \exp\left(2\ln(|C_nK_n|)-(1+\epsilon)\frac{2}{H_2(p)}\ln(|C_nK_n|)H_2(p)\right) \\
            &= \exp\bigl(-2\epsilon\ln(|C_nK_n|)\bigr) \\
            &= |C_nK_n|^{-2\epsilon}.
        \end{align*}
        Similarly, using both the upper and lower bounds on $|\bar{F}|$ from \eqref{eqn:coffee}, we see that
        \begin{align*}
            |C_nK_n||\bar{F}&\bar{F}^{-1}|\pi_2(p)^{\frac{1}{2}|\bar{F}|} = |\bar{F}\bar{F}^{-1}|\exp\bigl(\ln(|C_nK_n|)\bigr)\exp\left(-\frac{1}{2}|\bar{F}|H_2(p)\right) \\
            &\le |\bar{F}|^2\exp\left(\ln(|C_nK_n|)-\frac{1}{2}(1+\epsilon)\frac{2}{H_2(p)}\ln(|C_nK_n|)H_2(p)\right) \\
            &\le \alpha^2\ln(|C_nK_n|)^2\exp\bigl(-\epsilon\ln(|C_nK_n|)\bigr) \\
            &= \alpha^2\ln(|C_nK_n|)^2|C_nK_n|^{-\epsilon}.
        \end{align*}
        Combining these two estimates with \eqref{eqn:sugar} gives that
        \begin{align} \begin{split} \label{eqn:oak}
            \mu^n_p(\cI_{C_n,K_n}) &\ge 1 - \sum_{F \in \bar{\bF}_n} \left[ |C_nK_n|^{-2\epsilon} + \alpha^2\ln(|C_nK_n|)^2|C_nK_n|^{-\epsilon} \right] \\
            &= 1 - |\bar{\bF}_n|\Bigl[ |C_nK_n|^{-2\epsilon} + \alpha^2\ln(|C_nK_n|)^2|C_nK_n|^{-\epsilon} \Bigr] \\
            &= 1 - \frac{|\bar{\bF}_n|}{|C_nK_n|^{2\epsilon}} - \alpha^2\frac{|\bar{\bF}_n|\ln(|C_nK_n|)^2}{|C_nK_n|^\epsilon}.
            \end{split}
        \end{align}
        Note that we have $|\bar{\bF}_n| = |\bF_n|$ by construction. Now by \eqref{assump:few_oshapes}, we have that
        \begin{equation*}
           \lim_{n \to \infty} \frac{\ln(|\bar{\bF}_n|)}{\ln(|C_nK_n|)} =  \lim_{n \to \infty} \frac{\ln(|\bF_n|)}{\ln(|C_nK_n|)}  = 0.
        \end{equation*}
        Taking $\delta = \epsilon/2$, for all large enough $n$ we have that
        \begin{equation*}
            \frac{\ln(|\bar{\bF}_n|)}{\ln(|C_nK_n|)} \le \delta, 
        \end{equation*}
        which implies that $|\bar{\bF}_n| \le |C_nK_n|^\delta$.
        Using this estimate together with \eqref{eqn:oak}, we see that
        \begin{align*}
            1 \geq  \mu^n_p(\cI_{C_n,K_n}) &\ge 1 -  \frac{1}{|C_nK_n|^{2\epsilon - \delta}} - \alpha^2 \frac{\ln(|C_nK_n|)^2}{|C_nK_n|^{\epsilon-\delta}} \\
            &= 1 -  \frac{1}{|C_nK_n|^{3\epsilon/2}} - \alpha^2 \frac{\ln(|C_nK_n|)^2}{|C_nK_n|^{\epsilon/2}}.
        \end{align*}
        Letting $n$ tend to infinity and using that $|C_n K_n| \to \infty$ by \eqref{assump:CK_to_inf}, we obtain that
        \begin{equation*}
            \lim_{n \to \infty} \mu_p^n(\cI_{C_n,K_n}) = 1,
        \end{equation*}
        which concludes the proof of the theorem.
    \end{proof}
\end{repthm}

%%%%%%%%%%%%%%%%%%%%%%%%%%%%%%%%%%%%%%%%%%%%%%%%%%%%%%
%%%%%%%%%%%%%%%%%%%%%%%%%%%%%%%%%%%%%%%%%%%%%%%%%%%%%%
\section{Non-Identifiability results} \label{sect:neg}
%%%%%%%%%%%%%%%%%%%%%%%%%%%%%%%%%%%%%%%%%%%%%%%%%%%%%%
%%%%%%%%%%%%%%%%%%%%%%%%%%%%%%%%%%%%%%%%%%%%%%%%%%%%%%

In this section we aim to prove Theorem \ref{neg:regime}, which establishes sufficient conditions for non-identifiability a.a.s. 
Our approach to establishing non-identifiability is to show that certain \textit{blocking configurations} appear with high probability. To get the general idea, imagine that a pattern $w$ has two large subpatterns that are labeled identically except at their centers. The existence of these subpatterns would prevent successful reconstruction, because there would be no way to use the information from the reads to determine how the centers should be labeled. Thus, these subpatterns would serve as a blocking configuration. We note that analogous considerations have been used in \cite{motahari} and \cite{mossel} to establish non-identifiability results in the case $G = \Z^d$.
In order to establish that such blocking configurations appear with high probability, we employ a second moment argument.

In Section \ref{Sect:CombinatorialNegative} we establish some combinatorial lemmas related to these blocking configurations. The key result is Lemma \ref{neg:id_lem}, which shows that if a pattern contains a blocking configuration, then there is another pattern with the same reads. The goal of Section \ref{Sect:ExceptionalPatterns} is to bound the probability of the occurrence of some exceptional patterns that contain the blocking configurations but remain identifiable due to symmetry. Then in Section \ref{Sect:ProbNonId} we bound the probability of the existence of a blocking configuration from below, and finally in Section \ref{Sect:ProofNegative} we combine these results to prove Theorem \ref{neg:regime}.

%%%%%%%%%%%%%%%%%%%%%%%%%%%%%%%%%%%%%%%%%%%%%%%%%%%%%%%%%%
\subsection{Combinatorial results for non-identifiability} \label{Sect:CombinatorialNegative}
%%%%%%%%%%%%%%%%%%%%%%%%%%%%%%%%%%%%%%%%%%%%%%%%%%%%%%%%%%

We begin to make these ideas precise with the following definitions.

\begin{defn}
    Let $G$ be a group and let $C, K \subset G$. For any $h \in CK$, we define the \textit{center set around $h$} to be
    \begin{equation*}
        C_h = C \cap hK^{-1}.
    \end{equation*}
    In other words, $C_h$ is the set of centers $c \in C$ such that $h \in cK$. 
\end{defn}

\begin{defn}
    Let $G$ be a group and let $C, K \subset G$. For any $h \in CK$, we define the \textit{shell around $h$} to be
    \begin{equation*}
        S_h = (C_hK) \setminus \{h\}.
    \end{equation*}
    Additionally, we let $\bar{S}_h = C_h K$. 
\end{defn}

Notice that for arbitrary $g,h \in CK$, the shells $S_g$ and $S_h$ are not necessarily  %\revision{need not be}{are not necessarily} 
translates of one another, since the center set $C$ can be different \dquote{near} $g$ than it is \dquote{near} $h$. Nonetheless, for any $h \in CK$, we have that $C_h \subset hK^{-1}$, and hence $S_h \subset hK^{-1} K$. We refer to $h K^{-1} K$ as the \textit{maximal shell} around $h$. Since the shells around two elements need not be translates, it will be useful to have a formal definition for the \dquote{type} of shell around an element.

\begin{defn}
    Let $G$ be a group, let $C, K \subset G$, and let $a \in CK$. The \textit{shell type} at $a$ is defined to be $a^{-1}C_a$.
\end{defn}

Note that two elements $a$ and $b$ have the same shell type, i.e.,  $b^{-1}C_b = a^{-1}C_a$, if and only if $C_b = ba^{-1}C_a$. %Additionally, we can clearly form an equivalence relation on $CK$ according to the shell type of each element, and then we can  partition any subset of $CK$ into subsets of elements with the same shell type. More precisely, 
Given some set $A \subset CK$, we define $I_A$ to be the set of shell types that occur for some $a \in A$, and for any $\alpha \in I_A$, we let $A_{\alpha}$ denote the set of $a \in A$ with shell type $\alpha$. Then $|I_A|$ is the number of different shell types represented in the set $A$. We will consistently use lowercase Greek letters to represent elements of $I_A$ in order to avoid confusion with elements of $G$. 

The following elementary lemma is used to show that if a read shape $hK$ contains $g$, then $hK$ is entirely contained in the maximal shell around $g$.
\begin{lemma} \label{neg:trans_shell}
    Let $G$ be a group, let $g,h \in G$, and let $A \subset G$. If $g \in hA$, then $hA \subset gA^{-1}A$.
    \begin{proof}
        Suppose $g \in hA$, so $\exists a_1 \in A$ such that $g = ha_1$, and $h = ga_1^{-1}$. If we now consider some $c \in hA$ (possibly $g$), then $\exists a_2 \in A$ such that $c = ha_2 = ga_1^{-1}a_2$, which yields that $c \in gA^{-1}A$, and hence $hA \subset gA^{-1}A$.
    \end{proof}
\end{lemma}

%The main idea behind our proof of Theorem \ref{neg:regime} is to show that (under the hypotheses of the theorem) the probability of a random pattern having a repeated shell is large. %  determine the probability of a pattern having a repeated shell, meaning the pattern has the same labels at every element in the shells around two distinct elements. 
%We begin by defining the set of all patterns that have a repeated shell.
The blocking configurations that we use to guarantee non-identifiability are simply repeated shells. We begin our analysis of repeated shells by making some useful definitions.

\begin{defn}
    Let $G$ be a group and let $C,K \subset G$. Define the set $\cL_{C,K}$ to be the set of locations where we consider patterns to have repeated shells, i.e.,
    \begin{equation*}
        \cL_{C,K} = \bigl\{\{a,b\} \subset CK : C_b = ba^{-1}C_a, C_a \cap C_b = \varnothing\bigr\}.
    \end{equation*}
    
    Then, for any $a,b \in \cL_{C,K}$, we define the \textit{$(a,b)$-Repeated Shell} set, 
    \begin{equation*}
        \RS_{C,K}(a,b) = \bigl\{x \in \cA^G : x(S_a) = \sigma^{ba^{-1}}(x)(S_a) \text{ and } x(a) \ne x(b)\bigr\}.
    \end{equation*}
    Lastly, we define the \textit{Repeated Shell} set, 
    \begin{equation*}
        \RS_{C,K} = \bigcup_{\{a,b\} \in \cL_{C,K}} \RS_{C,K}(a,b).
    \end{equation*}
    Note that $\RS_{C,K}$ is indeed measurable. We view $\RS_{C,K}$ as the event that a random configuration has a repeated shell.
\end{defn}

Next, we define a collection of involutions on $\cA^G$ that swap the labels of two locations. We briefly discuss their utility following the definition.

\begin{defn}
    Let $G$ be a group and let $C,K \subset G$. Given two elements $a \ne b \in G$, we define the function $\varphi_{a,b}:\cA^G \to \cA^G$ by the rule
    \begin{equation*}
        \varphi_{a,b}(x)(g) = \begin{cases}
            x(b), & g = a \\
            x(a), & g = b \\
            x(g), & \text{otherwise.}
        \end{cases}
    \end{equation*}
     In other words, $\varphi_{a,b}(x)$ is the same as $x$, except that the labels at $a$ and $b$ are swapped.
\end{defn}

The idea here is that if a pattern $x$ contains repeated shells at $a$ and $b$, then $\varphi_{a,b}(x)$ will have the same multi-set of reads as $x$ does, which we state in the following lemma. It would be convenient for our arguments if having a repeated shell guaranteed non-identifiability. However, there may be some patterns with repeated shells that are still identifiable, and we account for such exceptional patterns in the next section. 

\begin{lemma} \label{neg:id_lem}
    Let $G$ be a group, let $C, K \subset G$, let $\{a,b\} \in \cL_{C,K}$, and let $x \in \RS_{C,K}(a,b)$. Then $\varphi_{a,b}(x) \ne x$ and %\revision{}{(make outer parens big)} 
    $\cR_{C,K}\bigl(x(CK)\bigr) = \cR_{C,K}\bigl(\varphi_{a,b}(x)(CK)\bigr)$.
    \begin{proof}
        Let $x' = \varphi_{a,b}(x)$, $w = x(CK)$, and $w' = x'(CK)$. First, clearly $w \ne w'$, because $x \in \RS_{C,K}(a,b)$, which implies that $w(a) \ne w(b)$, and thus $w'(a) = w(b) \ne w(a)$. To show that $\cR_{C,K}(w) = \cR_{C,K}(w')$, let us first find some permutation $\phi$ of $C$ such that $\forall c \in C$,
        \begin{equation*}
            \sigma^c(w)(K) = \sigma^{\phi(c)}(w')(K).
        \end{equation*}
        Note that $C_a \cap C_b = \varnothing$, since $\{a,b\} \in \cL_{C,K}$. Now let $\phi : C \to C$ be defined by the rule
        \begin{equation*}
            \phi(c) = \begin{cases}
                ba^{-1}c, & c \in C_a \\
                ab^{-1}c, & c \in C_b \\
                c, & \text{otherwise}.
            \end{cases}
        \end{equation*}
        To check that $\phi$ is well-defined, we note that if $c \in C_a$, then since $\{a,b\} \in \cL_{C,K}$, we must have
        \begin{equation*}
            ba^{-1}c \in ba^{-1}C_a = C_b,
        \end{equation*}
        and similarly for $c \in C_b$. 
        
        Additionally, we can clearly see that $\phi$ is an involution because any $c \in C_a$ maps to an element in $C_b$, and vice versa, and the maps between them are inverses of each other. Our goal now is to show that $\forall c \in C$,
        \begin{equation*}
            \sigma^c(w)(K) = \sigma^{\phi(c)}(w')(K).
        \end{equation*}
        
        First, for any $c_1 \in C \setminus (C_a \cup C_b)$ we have that $\phi(c_1) = c_1$. Then for any $k \in K$, we have
        \begin{equation*}
            \sigma^{c_1}(w)(k) = w(c_1k) = w'(c_1k) = \sigma^{c_1}(w')(k) = \sigma^{\phi(c_1)}(w')(k),
        \end{equation*}
        because $c_1k \ne a,b$. Therefore we have that $\sigma^{c_1}(w)(K) = \sigma^{\phi(c_1)}(w')(K)$.
        
        Next, we consider any $c_2 \in C_a$, which means $\phi(c_2) = ba^{-1}c_2$. We remark the following argument is identical for any $c_3 \in C_b$, replacing $a$ with $b$. The goal is similar as before, however we must now take into account if $c_2k = a$. Let us consider this case first, so let $k \in K$ such that $c_2k = a$ (as we know $a \in c_2K$), and then we have
        \begin{align*}
            \sigma^{c_2}(w)(k) &= w(c_2k) = w(a) \\
            &= w'(b) = w'(ba^{-1}a) = w'(ba^{-1}c_2k) \\
            &= \sigma^{ba^{-1}c_2}(w')(k) = \sigma^{\phi(c_2)}(w')(k).
        \end{align*}
        Next, we consider any other $k \in K$ such that $c_2k \ne a$ (and also $c_2k \ne b$). We now want to ensure that $\phi(c_2)k = ba^{-1}c_2k \ne a,b$.
        
        First, assume that $ba^{-1}c_2k = b$. This implies that $c_2k = a$ which has already been established to be false, so $ba^{-1}c_2k \ne b$. Now, assume that $ba^{-1}c_2k = a$. We also have that $\phi(c_2)k \in C_bK$, which means $a \in C_bK$. Let $c' \in C_b$ and $k' \in K$ such that $a = c'k'$, which implies that $c' = a(k')^{-1}$. Therefore, $c' \in C_a$, however this is a contradiction, because $C_a \cap C_b = \varnothing$, so $ba^{-1}c_2k \ne a$.
        
        Clearly $c_2k \in C_aK \setminus \{a\} = S_a$, and since $x \in \RS_{C,K}(a,b)$, we must have $w(c_2k) = w(ba^{-1}c_2k)$. Then, for any $g \ne a,b \in CK$ we have that $w(g) = w'(g)$, so
        \begin{align*}
            \sigma^{c_2}(w)(k) &= w(c_2k) = w(ba^{-1}c_2k) \\
            &= w'(ba^{-1}c_2k) \\
            &= \sigma^{ba^{-1}c_2}(w')(k) = \sigma^{\phi(c_2)}(w')(k).
        \end{align*}
        Combining the two previous displays, we have $\forall c_2 \in C_a$ that $\sigma^{c_2}(w)(K) = \sigma^{\phi(c_2)}(w')(K)$. Following the same argument for $c_3 \in C_b$, replacing $a$ with $b$, yields $\sigma^{c_3}(w)(K) = \sigma^{\phi(c_3)}(w')(K)$.
        
        As we have considered all elements in $C$, we therefore have $\forall c \in C$, $\sigma^c(w)(K) = \sigma^{\phi(c)}(w')(K)$ for the permutation $\phi$. Then, by Lemma \ref{read_eq}, we conclude that $\cR_{C,K}(w) = \cR_{C,K}(w')$.
    \end{proof}
\end{lemma}

%%%%%%%%%%%%%%%%%%%%%%%%%%%%%%%%%
\subsection{Exceptional patterns} \label{Sect:ExceptionalPatterns}
%%%%%%%%%%%%%%%%%%%%%%%%%%%%%%%%%

As we mention above, it is not necessarily true that any $x \in \RS_{C,K}(a,b)$ is non-identifiable; indeed, it may happen that $\varphi_{a,b}(x)(CK) \in \cC\bigl(x(CK)\bigr)$. In order to account for such patterns, we estimate the probability of selecting a pattern with a repeated shell that is still identifiable. To do so, we first require a few group theoretic definitions and lemmas. 

\begin{defn}
    Let $G$ be a group, and let $A \subset G$. The \textit{stabilizer} of $A$, denoted by $G_A$, is defined as
    \begin{equation*}
        G_A = \{g \in G : gA = A\}.
    \end{equation*}
\end{defn}

We remark that this is precisely the usual definition of the stabilizer if we consider the action of $G$ on the power set $\cP(G)$. It is an elementary exercise to show that $G_A$ is a subgroup of $G$. Since $G$ may be infinite, we cannot immediately guarantee that $G_A$ is a finite subgroup. However, the following lemma yields that $G_A$ must be finite whenever $A$ is finite.

\begin{lemma} \label{neg:stab_ord}
    Let $G$ be a group and let $A \subset G$ be finite. Then $|G_A| \le |A|$, and in particular, every element of $G_A$ has finite order.
    \begin{proof}
        Fix some element $a \in A$. Then clearly $|G_A| = |G_Aa|$. Since $gA = A$ for all $g \in G_A$, we observe that $ga \in A$ for all $g \in G_A$, and therefore $G_Aa \subset A$. Hence $|G_A| = |G_Aa| \le |A|$. Furthermore, since $G_A$ is a subgroup of $G$, every element in $G_A$ must have finite order.
    \end{proof}
\end{lemma}

For a set $F \subset G$, define the event $E(F)$ as the set of all patterns such that every element of $F$ has the same label. The measure of this event is given precisely by equation \eqref{prob:repeat_prob:eq1}. Next, we consider sufficient criteria for a pattern to be contained within two such events. 

\begin{lemma} \label{neg:part_orbit_lem}
    Let $G$ be a group, $C, K \subset G$, $a,b \in CK$,  and let $x \in \cA^G$ such that $x(a) \ne x(b)$. Further suppose that there is some $g \in G_C$ such that $x(CK) = \sigma^g\bigl(\varphi_{a,b}(x)\bigr)(CK)$. % \revision{}{(make previous outer parens big)}. 
    Then $\exists m \in \bigr(0,\ord(g)\bigl)$ such that $b = g^ma$ and $x \in E(P_1) \cap E(P_2)$, where
    \begin{equation*}
        P_1 = \{a, ga, \cdots, g^{m-1}a\}, \quad \text{and} \quad P_2 = \{b, gb, \cdots, g^{\ord(g)-m-1}b\}.
    \end{equation*}
    \begin{proof}
        Let $x' = \varphi_{a,b}(x)$. By Lemma \ref{neg:stab_ord}, $g$ has finite order. We note that
        \begin{equation} \label{eqn:jano}
            x(h) = \sigma^g(x')(h) = x'(gh), \quad \forall h \in CK.     
        \end{equation}
        Additionally, since $x(a) \ne x(b)$, we see that $x \ne x'$, and therefore $g \ne e$. Also, since $x(a) \ne x(b)$, we note that $a \ne b$.
        
        Now, for any $0 < n < \ord(g)$, we have that $g^n \ne e$, and then $g^na \neq a$. By definition of $x'$, if $g^na \ne b$, then we have
        \begin{equation} \label{eqn:lighthouses}
            x'(g^na) = x(g^na).
        \end{equation}
        Also, since $a \in CK$, there exists some $c \in C$ and $k \in K$ such that $a = ck$. Since $g \in G_C$, we have that $g^nc \in C$, and hence $g^na = g^nck \in CK$. Then by \eqref{eqn:jano}, we have
        \begin{equation} \label{eqn:seahorses}
            x(g^na) = \sigma^g(x')(g^na) = x'(g^{n+1}a).
        \end{equation}
        %by \eqref{eqn:jano}.
        
        Suppose for contradiction that $\forall n \in \bigl(0,\ord(g)\bigr)$, $b \ne g^na$, and consider any such $n$. By the inductive application of \eqref{eqn:lighthouses} and \eqref{eqn:seahorses}, we obtain
        \begin{equation*}
            x(a) = x'(ga) = x(ga) = x'(g^2a) = \cdots = x(g^{\ord(g)-1}a) = x'(g^{\ord(g)}a) = x'(a).
        \end{equation*}
        However, this equality contradicts our assumption that $x(a) \ne x(b)$, because $ x'(a) = x(b)$ by definition. Therefore $\exists m \in (0,\ord(g))$ such that $b = g^ma$. 
        
        The inductive application of \eqref{eqn:lighthouses} and \eqref{eqn:seahorses} starting from $x(a)$ now gives
        \begin{equation*}
            x(a) = x'(ga) = x(ga) = \cdots = x(g^{m-1}a) = x'(g^ma) = x'(b) = x(a).
        \end{equation*}
        This implies $x \in E(P_1)$. Starting with $x(b)$ and inductively applying \eqref{eqn:lighthouses} and \eqref{eqn:seahorses}, we have
        \begin{equation*}
            x(b) = x'(gb) = \cdots = x(g^{\ord(g)-m-1}b) = x'(g^{\ord(g)-m}b) = x'(a) = x(b),
        \end{equation*}
        which implies $x \in E(P_2)$. Therefore, $x \in E(P_1) \cap E(P_2)$.
    \end{proof}
\end{lemma}

Next we give a bound on the probability of $\RS_{C,K} \cap \cI_{C,K}$, the event that a pattern has a repeated shell and yet is still identifiable.

\begin{lemma} \label{neg:insct_bound}
    Let $G$ be a group, let $C,K \subset G$, let $\cA$ be a finite alphabet, let $p$ be a probability vector on $\cA$, and let $\mu_p$ be the product measure $p^G$ on $\cA^G$. Then
    \begin{equation*}
        \mu_p\bigl(\RS_{C,K} \cap \cI_{C,K}\bigr) \le |CK|^3 \pi_2(p)^{\frac{|CK|}{2}-1}.
    \end{equation*}
    \begin{proof}
        %\revision{}{(throughout proof, make outer parens big when nested)}
        First, for $a,b \in CK$ and $g \in G_C$, define the set %$\cT_{C,K}(a,b,g)$ as
        \begin{equation*}
            \cT_{C,K}(a,b,g) = \bigl\{ x \in \cA^G : x(a) \ne x(b), x(CK) = \sigma^g\bigl( \varphi_{a,b}(x)\bigr)(CK) \bigr\},
        \end{equation*}
        and then define the set %$\cT_{C,K}$ as
        \begin{equation*}
            \cT_{C,K} = \bigcup_{a,b \in CK} \bigcup_{g \in G_C} \cT_{C,K}(a,b,g).
        \end{equation*}
        
        Consider some $x \in \RS_{C,K} \cap \cI_{C,K}$. Then there exists $\{a,b\} \in \cL_{C,K}$ such that $x \in \RS_{C,K}(a,b)$, and in particular $x(a) \ne x(b)$. Then by Lemma \ref{neg:id_lem}, $x' = \varphi_{a,b}(x) \ne x$ and $\cR_{C,K}\bigl(x(CK)\bigr) = \cR_{C,K}\bigl(x'(CK)\bigr)$. Since we also have that $x \in \cI_{C,K}$, it must be the case that $x'(CK) \in \cC\bigl(x(CK)\bigr)$, and therefore $\exists g \in G_C$ such that $\sigma^g(x')(CK) = x(CK)$. Hence $x \in \cT_{C,K}(a,b,g)$. Since $x \in \RS_{C,K} \cap \cI_{C,K}$ was arbitrary, we see that
        \begin{equation*}
            \RS_{C,K} \cap \cI_{C,K} \subset \cT_{C,K},
        \end{equation*}
        and therefore $\mu_p\bigl(\RS_{C,K} \cap \cI_{C,K}\bigr) \le \mu_p(\cT_{C,K})$ by monotonicity.
        
        %\revision{Below: $w$}{$x$}
        We now focus on finding a bound for $\mu_p\bigl(\cT_{C,K}(a,b,g)\bigr)$. For any $a,b \in CK$ and $g \in G_C$, suppose $x \in \cT_{C,K}(a,b,g)$, and let $x' = \varphi_{a,b}(x)$. By definition, $x$ satisfies the hypotheses of Lemma \ref{neg:part_orbit_lem}, and therefore (by the lemma) there is some integer $m \in (0,\ord(g))$ such that $b = g^ma$ and $x \in E(P_1) \cap E(P_2)$, with
        \begin{equation*}
            P_1 = \{a, ga, \cdots, g^{m-1}a\}, \quad P_2 = \{b, gb, \cdots, g^{\ord(g)-m-1}b\}.
        \end{equation*}
        Note that $P_1 \cup P_2 = \langle g \rangle a$, and this union contains both $a$ and $b$. Since the only difference between $x$ and $x'$ is the labels at $a$ and $b$, we have that $x'\bigl( CK \setminus \{a,b\} \bigr) = x\bigl( CK \setminus \{a,b\} \bigr)$, and in particular $x'\bigl(CK \setminus \langle g \rangle a\bigr) = x\bigl(CK \setminus \langle g \rangle a\bigr)$. Since we also have that $x(CK) = \sigma^g(x')(CK)$ and translation by $g$ preserves the set $CK \setminus \langle g \rangle a$, it must be the case that
        \begin{equation*}
            x\bigl(CK \setminus \langle g \rangle a\bigr) = \sigma^g(x')\bigl(CK \setminus \langle g \rangle a\bigr) = \sigma^g(x)\bigl(CK \setminus \langle g \rangle a\bigr).
        \end{equation*}
        This equality establishes that $x$ is in the repeat set $\cR\bigl((CK\setminus \langle g \rangle a); g\bigr)$ (defined in Section \ref{sect:prob}). Thus, we have shown that 
        \begin{equation} \label{eqn:spynotes}
            \cT_{C,K}(a,b,g) \subset E(P_1) \cap E(P_2) \cap \cR\bigl((CK\setminus \langle g \rangle a); g\bigr).
        \end{equation}
        
        Since $P_1$, $P_2$, and $CK \setminus \langle g \rangle a$ are all pairwise disjoint, we have that $E(P_1)$, $E(P_2)$, and $\cR(CK\setminus \langle g \rangle a; g)$ are mutually independent, and then by \eqref{eqn:spynotes} and \eqref{prob:fact2},
        \begin{align*}
            \mu_p\bigl(\cT_{C,K}(a,b,g)\bigr) &\le \mu_p\bigl(E(P_1)\cap E(P_2)\cap \cR(CK\setminus \langle g \rangle a; g)\bigr) \\
            &= \mu_p\bigl(E(P_1)\bigr)\mu_p\bigl(E(P_2)\bigr)\mu_p\bigl(\cR(CK\setminus \langle g \rangle a; g)\bigr).
        \end{align*}
        Furthermore, by \eqref{prob:repeat_prob:eq1} and Theorem \ref{prob:repeat_bounds}, we have that
        \begin{align*}
            \mu_p\bigl(\cT_{C,K}(a,b,g)\bigr) &\le \pi_{|P_1|}(p)\pi_{|P_2|}(p)\pi_2(p)^{\frac{1}{2}|CK \setminus \langle g \rangle a|} \\
            &= \pi_m(p)\pi_{\ord(g)-m}(p)\pi_2(p)^{\frac{1}{2}(|CK|-\ord(g))}.
        \end{align*}
        
        Next we distinguish several cases in terms of $m$ and $\ord(g)$. If $\ord(g) = 2$, then it must be the case that $m = 1$, since $0 < m < 2$ and $m \in \N$, and in this case,
        \begin{align*}
            \mu_p\bigl(\cT_{C,K}(a,b,g)\bigr) &\le \pi_1(p)\pi_{2-1}(p)\pi_2(p)^{\frac{|CK|}{2}-\frac{2}{2}} \\
            &= \pi_2(p)^{\frac{|CK|}{2}-1},
        \end{align*}
        since $\pi_1(p) = 1$. Next, in the case that either $m = 1$ or $\ord(g)-m = 1$, the other quantity is clearly $\ord(g)-1$ with $\ord(g) \ge 3$. Applying Lemma \ref{pvec:pi_upper}, we get
        \begin{align*}
            \mu_p\bigl(\cT_{C,K}(a,b,g)\bigr) &\le \pi_1(p)\pi_{\ord(g)-1}(p)\pi_2(p)^{\frac{|CK|}{2}-\frac{\ord(g)}{2}} \\
            &\le \pi_2(p)^{\frac{\ord(g)}{2}-\frac{1}{2}}\pi_2(p)^{\frac{|CK|}{2}-\frac{\ord(g)}{2}} \\
            &\le \pi_2(p)^{\frac{|CK|}{2}-1}.
        \end{align*}
        Lastly, in the case that both $m \ge 2$ and $\ord(g)-m \ge 2$, we also have by Lemma \ref{pvec:pi_upper} that
        \begin{align*}
            \mu_p\bigl(\cT_{C,K}(a,b,g)\bigr) &\le \pi_m(p)\pi_{\ord(g)-m}(p)\pi_2(p)^{\frac{|CK|}{2}-\frac{\ord(g)}{2}} \\
            &\le \pi_2(p)^{\frac{m}{2}}\pi_2(p)^{\frac{\ord(g)-m}{2}}\pi_2(p)^{\frac{|CK|}{2}-\frac{\ord(g)}{2}} \\
            &\le \pi_2(p)^{\frac{|CK|}{2}-1}.
        \end{align*}
        In all cases, we see that for any $a,b \in CK$ and $g \in G_C$,
        \begin{equation*}
            \mu_p\bigl(\cT_{C,K}(a,b,g)\bigr) \le \pi_2(p)^{\frac{|CK|}{2}-1}.
        \end{equation*}
        
        Finally, by applying the union bound, the previous display, and the fact that $|G_C| \leq |C| \leq |CK|$ (by Lemma \ref{neg:stab_ord}), we have
        \begin{align*}
            \mu_p(\cT_{C,K}) &\le \sum_{a,b \in CK} \sum_{g \in G_C} \mu_p\bigl(\cT_{C,K}(a,b,g)\bigr) \\
            &\le \sum_{a,b \in CK} \sum_{g \in G_C} \pi_2(p)^{\frac{|CK|}{2}-1} \\
            &= |CK|^2 |G_C| \pi_2(p)^{\frac{|CK|}{2}-1} \\
            &\le |CK|^3 \pi_2(p)^{\frac{|CK|}{2}-1}.
        \end{align*}
        
        As we have already shown that $\mu_p(\RS_{C,K} \cap \cI_{C,K}) \le \mu_p(\cT_{C,K})$, we have established the desired result.
    \end{proof}
\end{lemma}

%%%%%%%%%%%%%%%%%%%%%%%%%%%%%%%%%%%%%%%%%%%%%%%%%%%%%%%%%%%%%%%%%%%
\subsection{Estimating the probability of non-identifiability} \label{Sect:ProbNonId}
%%%%%%%%%%%%%%%%%%%%%%%%%%%%%%%%%%%%%%%%%%%%%%%%%%%%%%%%%%%%%%%%%%%

We now turn towards bounding the probability of $\RS_{C,K}$ from below. At its core, our argument involves an application of the second moment method to collections of disjoint shells. Using this argument, we show that the probability of having a repeated shell is large, which then implies that the probability of non-identifiability is large. We begin with a definition. 

\begin{defn} \label{def:DSC}
    Let $G$ be a group, and let $C, K \subset G$. Then a set $D \subset CK$ satisfies the \textit{Disjoint Shell Condition} (DSC) if for any $a \ne b \in D$, $\bar{S}_a \cap \bar{S}_b = \varnothing$. 
\end{defn}
Note that if $D$ satisfies the DSC, then for any distinct $a,b \in D$ with the same shell type, we have $\{a,b\} \in \cL_{C,K}$. 
Recall that for an arbitrary set $A \subset CK$, the set $I_A$ %\revision{$I_D$}{$I_A$} 
is the set of shell types of elements of $A$,  %\revision{$D$}{$A$}, 
and for $\alpha \in I_A$, % \revision{$\alpha \in I_D$}{$\alpha \in I_A$}, 
the set %\revision{$D_{\alpha}$}{$A_\alpha$} 
is the set of all elements of $A$ %\revision{$D$}{$A$} 
with shell type $\alpha$.
\begin{lemma} \label{neg:id_prob}
    Let $G$ be a group, and let $C, K \subset G$. Suppose there is a nonempty set $D \subset CK$ satisfying the DSC such that $\forall \alpha \in I_D$, $|D_\alpha| \ge 2$. Then
    \begin{equation*}
        \mu_p(\RS_{C,K}) \ge 1 - \frac{4\sqrt{|I_D|}(1-\pi_2(p))^{-1}}{|D|\pi_2(p)^{\frac{1}{2}|K^{-1}K|}}\left[\frac{\sqrt{|I_D|}}{|D|\pi_2(p)^{\frac{1}{2}|K^{-1}K|}} + 2\right].
    \end{equation*}
    \begin{proof}
        To begin, $\forall \alpha \in I_D$, we let $S_\alpha = a^{-1}S_a$ for any $a \in D_\alpha$. Note that $S_\alpha$ is the same regardless of the choice of $a$, because for any $a,b \in D_\alpha$, we have $a^{-1}C_a = b^{-1}C_b$.
        
        Now, for any $\alpha \in I_D$ and distinct $a, b \in D_\alpha$, we define the indicator random variable $\1_{a,b} = \1_{b,a}$ for an arbitrary pattern $x \in \cA^G$ by the rule
        \begin{equation*}
            \1_{a,b}(x) = \begin{cases}
                1, & x \in \cR(S_a; ba^{-1}) \cap \left(\cA^{G} \setminus E(\{a,b\})\right) \\
                0, & \text{otherwise},
            \end{cases}
        \end{equation*}
        where $E(\{a,b\})$ denotes the event that an arbitrary pattern has the same label at $a$ and $b$. Note that by definition, $\1_{a,b}(x) = 1$ if and only if $x \in \RS_{C,K}(a,b)$.
        
        For any $\alpha \in I_D$, we define the set $D_\alpha^*$ to be the collection of two-element subsets of $D_\alpha$, and we note that $|D_\alpha^*| = \binom{|D_\alpha|}{2}$. We then define the random variable $X$ by
        \begin{equation*}
            X = \sum_{\alpha \in I_D} \sum_{\{a,b\} \in D_\alpha^*} \1_{a,b}.
        \end{equation*}
        Clearly, if $X(x) \ge 1$ for some $x \in \cA^G$, then we have some $\{a,b\} \in \cL_{C,K}$ and $x \in \RS_{C,K}(a,b)$, which means that $x \in \RS_{C,K}$. Therefore $(X \ge 1) \subset \RS_{C,K}$, and we can bound the probability of $\RS_{C,K}$ by monotonicity, i.e.,
        \begin{equation*}
            \mu_p(\RS_{C,K}) \ge \mu_p(X \ge 1) = 1 - \mu_p(X < 1).
        \end{equation*}
        Since $X$ is an integer valued random variable, the event that $X < 1$ is the same as the event that $X \le 0$. We then have by Chebyshev's Inequality that
        \begin{align} \begin{split} \label{eqn:greatscott}
            \mu_p(\RS_{C,K}) &\ge 1 - \mu_p(X \le 0) \\
            &= 1 - \mu_p\bigl(\E[X] - X \ge \E[X]\bigr) \\
            &= 1 - \mu_p\left(\E[X] - X \ge \frac{\E[X]}{\sqrt{\Var(X)}}\sqrt{\Var(X)}\right) \\
            &\ge 1 - \frac{\Var(X)}{\E[X]^2}.
            \end{split}
        \end{align}
        We now focus on finding an expression for $\Var(X)$ and $\E[X]$. First, we have that
        \begin{align*}
            \E[X] &= \sum_{\alpha \in I_D} \sum_{\{a,b\} \in D_\alpha^*} \E[\1_{a,b}] \\
            &= \sum_{\alpha \in I_D} \sum_{\{a,b\} \in D_\alpha^*} \mu_p\Bigl(\cR(S_a;ba^{-1}) \cap \bigl(\cA^G \setminus E(\{a,b\})\bigr)\Bigr).
        \end{align*}
        Note that $S_a \cup ba^{-1}S_a = S_a \cup S_b$. Since we know that $\bar{S}_a \cap \bar{S}_b = \varnothing$ (because $D$ satisfies the DSC), this union must be disjoint. Additionally, we know that $\{a,b\}$ and $S_a \cup S_b$ are disjoint, so by \eqref{prob:fact2}, we see that $\cR(S_a; ba^{-1})$ and $E(\{a,b\})$ are independent, and hence $\cR(S_a; ba^{-1})$ and $\cA^G \setminus E(\{a,b\})$ (the complement of $E(\{a,b\})$) are independent as well. Using this fact in the previous display, we obtain
        \begin{align*}
            \E[X] &= \sum_{\alpha \in I_D} \sum_{\{a,b\} \in D_\alpha^*} \mu_p\bigl(\cR(S_a;ba^{-1})\bigr)\mu_p\bigl(\cA^G \setminus E(\{a,b\})\bigr) \\
            &= \sum_{\alpha \in I_D} \sum_{\{a,b\} \in D_\alpha^*} \mu_p\bigl(\cR(S_a;ba^{-1})\bigr)\left(1 - \mu_p\bigl(E(\{a,b\})\bigr)\right).
        \end{align*}
        Since $S_a$ and $ba^{-1} = S_b$ are disjoint, Theorem \ref{prob:n_disj_repeat_prob} gives that $\mu_p\bigl(\cR(S_a;ba^{-1})\bigr) = \pi_2(p)^{|S_a|} = \pi_2(p)^{|S_\alpha|}$. Additionally, \eqref{prob:repeat_prob:eq1} gives that $\mu_p\bigl(E(\{a,b\})\bigr) = \pi_2(p)$. Substituting these expressions into the previous display yields
        \begin{align*}
            \E[X] &= \sum_{\alpha \in I_D} \sum_{\{a,b\} \in D_\alpha^*} \pi_2(p)^{|S_\alpha|}(1-\pi_2(p)) \\
            &= (1-\pi_2(p)) \sum_{\alpha \in I_D} \binom{|D_\alpha|}{2} \pi_2(p)^{|S_\alpha|}.
        \end{align*}
        Since $|D_\alpha| \ge 2$, we also have that $|D_\alpha| - 1 \ge \frac{|D_\alpha|}{2}$, which gives the inequality $\binom{|D_\alpha|}{2} = \frac{|D_\alpha|(|D_\alpha|-1)}{2} \ge \frac{|D_\alpha|^2}{4}$. Applying this inequality to the previous display gives
        \begin{equation} \label{eqn:greatexpectation}
            \E[X] \ge \frac{1-\pi_2(p)}{4} \sum_{\alpha \in I_D} |D_\alpha|^2\pi_2(p)^{|S_\alpha|}.
        \end{equation} 
        
        Now we consider $\Var(X)$, which can be written in terms of covariances,
        \begin{equation*}
            \Var(X) = \sum_{\alpha,\beta \in I_D} \sum_{\{a,b\} \in D_\alpha^*} \sum_{\{c,d\} \in D_\beta^*} \Cov(\1_{a,b},\1_{c,d}).
        \end{equation*}
        If we consider two distinct shell types $\alpha,\beta \in I_D$ and pairs $\{a,b\} \in \alpha$ and $\{c,d\}$ in $\beta$, then clearly $\{c,d\} \cap \{a,b\} = \varnothing$, because $c$ and $d$ are in a different equivalence class from $a$ and $b$. Additionally, since $D$ satisfies the DSC, the shells around $a,b,c$, and $d$ are all pairwise disjoint, which means $\1_{a,b}$ and $\1_{c,d}$ are independent. As a result, $\Cov(\1_{a,b}, \1_{c,d}) = 0$. Dropping these terms in our expression for $\Var(X)$ yields 
        \begin{equation*}
            \Var(X) = \sum_{\alpha \in I} \sum_{\{a,b\} \in D_\alpha^*} \sum_{\{c,d\} \in D_\alpha^*} \Cov(\1_{a,b},\1_{c,d}).
        \end{equation*}
        The remaining terms in this sum can be split into three cases. 
        
        The first is when $\{c,d\} = \{a,b\}$. In this case, 
        \begin{equation*}
            \Cov(\1_{a,b},\1_{a,b}) = \E[\1_{a,b}]-\E[\1_{a,b}]^2 \le \E[\1_{a,b}] = \pi_2(p)^{|S_\alpha|}(1-\pi_2(p)).
        \end{equation*}
        Note that this case clearly only occurs once in the innermost sum. Next, we have the case that $|\{a,b\} \cup \{c,d\}| = 3$, meaning one of $c$ or $d$ is exactly $a$ or $b$. Without loss of generality, suppose that $c = b$. To account for this, we estimate the number of such pairs of two-element subsets of $D_\alpha$ by $2(|D_\alpha|-2) \le 2|D_\alpha|$, because for $c$, we choose one of $a$ or $b$ (so we have 2 choices), and for $d$ we choose some other element in $D_\alpha$, of which there are $|D_\alpha|-2$. In this case, we have that
        \begin{equation*}
            \Cov(\1_{a,b},\1_{b,d}) = \E[\1_{a,b}\1_{b,d}] - \E[\1_{a,b}]\E[\1_{b,d}] \le \E[\1_{a,b}\1_{b,d}].
        \end{equation*}
        Since $\1_{a,b}$ and $\1_{b,d}$ are indicator variables, the expectation of their product is just the probability of the intersection of the events on which they are both $1$, and so %\revision{}{(make outer parens of $\cA^G\setminus E(\{a,b\})$ big)}
        \begin{equation*}
            \E[\1_{a,b}\1_{b,d}] \le \mu_p\Bigl(\cR(S_a;ba^{-1})\cap\bigl(\cA^G \setminus E(\{a,b\})\bigr)\cap \cR(S_b; db^{-1})\Bigr).
        \end{equation*}
        We remark that we have left the event $\cA^G \setminus E(\{b,d\})$ out of this intersection, which gives rise to the inequality. Noting that the event $\cR(S_a;ba^{-1}) \cap \cR(S_b; db^{-1})$ denotes the event that $S_a$, $S_b$, and $S_d$ all have the same labels, with $ba^{-1}$ mapping $S_a$ to $S_b$ and $da^{-1}$ mapping $S_a$ to $S_d$, we can equivalently represent this event as $\cR(S_a; ba^{-1}, da^{-1})$. Substituting, we get %\revision{}{(modify paren size below)}
        \begin{equation*}
            \E[\1_{a,b}\1_{b,d}] \le \mu_p\Bigl(\cR(S_a; ba^{-1}, da^{-1}) \cap \bigl(\cA^G \setminus E(\{a,b\})\bigr)\Bigr).
        \end{equation*}
        Since we have that $S_a \cup S_b \cup S_d$ is disjoint from $\{a,b\}$, we have by \eqref{prob:fact2} that $\cR(S_a; ba^{-1}, da^{-1})$ and $E(\{a,b\})$ are independent, which means $\cR(S_a; ba^{-1}, da^{-1})$ and $\cA^G \setminus E(\{a,b\})$ are also independent. By Theorem \ref{prob:n_disj_repeat_prob}, $\mu_p\bigl(\cR(S_a;ba^{-1},da^{-1})\bigr) = \pi_3(p)^{|S_a|} = \pi_3(p)^{|S_\alpha|}$, and we know $\mu_p\bigl(E(\{a,b\})\bigr) = \pi_2(p)$, which yields %\revision{}{(modify paren size below)}
        \begin{align*}
            \E[\1_{a,b}\1_{b,d}] &\le \mu_p\bigl(\cR(S_a; ba^{-1}, da^{-1})\bigr)\Bigl(1 - \mu_p\bigl(E(\{a,b\})\bigr)\Bigr) \\
            &= \pi_3(p)^{|S_\alpha|}(1-\pi_2(p)).
        \end{align*}
        Thus, by applying Lemma \ref{pvec:pi_upper}, we obtain that
        \begin{equation*}
            \Cov(\1_{a,b},\1_{b,d}) \le \pi_3(p)^{|S_\alpha|}(1-\pi_2(p)) \le \pi_2(p)^{\frac{3}{2}|S_\alpha|}(1-\pi_2(p)).
        \end{equation*}
        The third and final case concerns the pairs of disjoint two element subsets of $D_\alpha$. In this case, each of the four shells are disjoint, and thus $\1_{a,b}$ and $\1_{c,d}$ are independent and all of these terms are zero. 
        
        By applying our estimates in all three cases, we obtain that 
        \begin{align} \begin{split} \label{eqn:greatvar}
            \Var(X) &\le \sum_{\alpha \in I_D} \sum_{\{a,b\} \in D_\alpha^*} \pi_2(p)^{|S_\alpha|}(1-\pi_2(p))\left[1 + 2|D_\alpha|\pi_2(p)^{\frac{1}{2}|S_\alpha|}\right] \\
            &= (1-\pi_2(p)) \sum_{\alpha \in I_D} \binom{|D_\alpha|}{2}\pi_2(p)^{|S_\alpha|} \left[1 + 2|D_\alpha|\pi_2(p)^{\frac{1}{2}|S_\alpha|}\right] \\
            &\le \frac{(1-\pi_2(p))}{4} \sum_{\alpha \in I_D} |D_\alpha|^2\pi_2(p)^{|S_\alpha|}\left[1 + 2|D_\alpha|\pi_2(p)^{\frac{1}{2}|S_\alpha|}\right].
            \end{split}
        \end{align}
        
        By combining \eqref{eqn:greatscott}, \eqref{eqn:greatexpectation}, and \eqref{eqn:greatvar}, we get
        \begin{equation*}
            \mu_p(\RS_{C,K}) \ge 1 - \frac{4}{(1-\pi_2(p))}\sum_{\alpha \in I_D} \frac{|D_\alpha|^2 \pi_2(p)^{|S_\alpha|}\left[1 + 2|D_\alpha|\pi_2(p)^{\frac{1}{2}|S_\alpha|}\right]}{\left[\sum_{\beta \in I_D} |D_\beta|^2 \pi_2(p)^{|S_\beta|}\right]^2}.
        \end{equation*}
        Let us now simplify the expression on the right-hand side. To compress the notation, let 
        \begin{equation*}
            P_\alpha = |D_\alpha|\pi_2(p)^{\frac{|S_\alpha|}{2}}.
        \end{equation*}
        Then we see that
        \begin{align} \begin{split} \label{eqn:closer}
            \mu_p(\RS_{C,K}) &\ge 1 - \frac{4}{(1-\pi_2(p))}\frac{\sum_{\alpha \in I_D} P_\alpha^2}{\left[\sum_{\beta \in I_D} P_\beta^2\right]^2} - \frac{8}{(1-\pi_2(p))} \frac{\sum_{\alpha \in I_D} P_\alpha^3}{\left[\sum_{\beta \in I_D} P_\beta^2\right]^2} \\
            &\stackrel{(i)}{\ge} 1 - \frac{4(1-\pi_2(p))^{-1}}{\sum_{\alpha \in I_D} P_\alpha^2} - \frac{8}{(1-\pi_2(p))} \frac{\left[\sum_{\alpha \in I_D} P_\alpha^2\right]^\frac{3}{2}}{\left[\sum_{\beta \in I_D} P_\beta^2\right]^2} \\
            &= 1 - \frac{4(1-\pi_2(p))^{-1}}{\sum_{\alpha \in I_D} P_\alpha^2} - \frac{8(1-\pi_2(p))^{-1}}{\left[\sum_{\alpha \in I_D} P_\alpha^2\right]^\frac{1}{2}}.
            \end{split}
        \end{align}
        Inequality $(i)$ is possible as a result of the monotonicity of the $\ell^p$ norms: if we consider the vector $P = (P_\alpha)_{\alpha \in I_D}$, then
        \begin{equation*}
            \sum_{\alpha \in I_D} P_\alpha^3 = (\lVert P \rVert_3)^3 \le (\lVert P \rVert_2)^3 = \left(\sum_{\alpha \in I_D} P_\alpha^2\right)^{\frac{3}{2}}.
        \end{equation*}
        We now focus on bounding from below the sum that appears in this expression. Given any $a \in CK$, we have that $\bar{S}_a \subset aK^{-1}K$, so $\forall \alpha \in I_D$, $|S_\alpha| \le |K^{-1}K|$. Thus
        \begin{align} \begin{split} \label{eqn:closest}
            \sum_{\alpha \in I_D} P_\alpha^2 &= \sum_{\alpha \in I_D} |D_\alpha|^2\pi_2(p)^{|S_\alpha|} \\
            &\ge \sum_{\alpha \in I_D} |D_\alpha|^2\pi_2(p)^{|K^{-1}K|} \\
            &= \pi_2(p)^{|K^{-1}K|} \sum_{\alpha \in I_D} |D_\alpha|^2 \\
            &\stackrel{(ii)}{\ge} \pi_2(p)^{|K^{-1}K|} \frac{1}{|I_D|} \left[\sum_{\alpha \in I_D} |D_\alpha| \right]^2 \\
            &\stackrel{(iii)}{=} \pi_2(p)^{|K^{-1}K|} \frac{|D|^2}{|I_D|}.
            \end{split}
        \end{align}
        Equality $(iii)$ is due to the fact that $\sum_{\alpha \in I_D} |D_\alpha| = |D|$. Inequality $(ii)$ follows from the Cauchy-Schwarz inequality:
        \begin{equation*}
            \left(\sum_{\alpha \in I_D} 1\cdot|D_\alpha|\right)^2 \le \sum_{\alpha \in I_D} 1^2 \sum_{\alpha \in I_D} |D_\alpha|^2 = |I_D| \sum_{\alpha \in I_D} |D_\alpha|^2
        \end{equation*}
        Finally, applying the estimate \eqref{eqn:closest} in \eqref{eqn:closer} and then simplifying gives
        \begin{align*}
            \mu_p(\RS_{C,K}) &\ge 1 - \frac{4|I_D|(1-\pi_2(p))^{-1}}{|D|^2\pi_2(p)^{|K^{-1}K|}} - \frac{8\sqrt{|I_D|}(1-\pi_2(p)^{-1})}{|D|\pi_2(p)^{\frac{1}{2}|K^{-1}K|}} \\
            &= 1 - \frac{4\sqrt{|I_D|}(1-\pi_2(p))^{-1}}{|D|\pi_2(p)^{\frac{1}{2}|K^{-1}K|}}\left[\frac{\sqrt{|I_D|}}{|D|\pi_2(p)^{\frac{1}{2}|K^{-1}K|}} + 2\right],
        \end{align*}
        which concludes the proof.
    \end{proof}
\end{lemma}

%%%%%%%%%%%%%%%%%%%%%%%%%%%%%%%%%%%%%%%%%%%%%%
\subsection{Proof of Theorem \ref{neg:regime}} \label{Sect:ProofNegative}
%%%%%%%%%%%%%%%%%%%%%%%%%%%%%%%%%%%%%%%%%%%%%%

With this probability bound on $\RS_{C,K}$, we can now prove Theorem \ref{neg:regime}. We note that the theorem does not require defining a sequence of sets satisfying the DSC (Definition \ref{def:DSC}), as that could be cumbersome for an arbitrary choice of $G$, $C$, and $K$. It turns out, however, that we can construct a sufficiently large set satisfying the DSC out of a set of elements with potentially intersecting shells using the set $B_n$. Recall the statement of the theorem.

\begin{repthm}{neg:regime}
    Let $\cA$, $p$, $\{G_n\}$, $\{C_n\}$, $\{K_n\}$, and $\{\mu_p^n\}$ be as in Section \ref{Sect:Intro} such that $\lim_n |C_nK_n| = \infty$ \eqref{assump:CK_to_inf}. Suppose there exists a sequence $\{B_n\}$ such that $B_n \subset C_nK_n$,
    \begin{itemize}[itemsep=0.4\baselineskip]
        \item[\eqref{assump:B_few_shells}] $\lim_n \frac{\ln|I_{B_n}|}{\ln|C_nK_n|} = 0$, and
        \item[\eqref{assump:B_large}] $\lim_n \frac{\ln|B_n|}{\ln|C_nK_n|} = 1$.
    \end{itemize}
    Further, suppose
    \begin{itemize}
        \item[\eqref{assump:shell_size_bound}] $\exists \epsilon >0 \text{ s.t. } \forall n \gg 0 \quad |K^{-1}_nK_n| \leq (1-\epsilon)\frac2{H_2(p)}\ln |C_nK_n|$.
    \end{itemize}
    Then non-identifiability holds a.a.s.
    \begin{proof}
        First, we take $n$ large enough that $|C_nK_n| \ge 2$. We begin by inductively constructing a sequence of sets $\{D_n\}$ that each satisfies the DCS with respect to $G_n$, $C_n$, and $K_n$. For the first element $d_1$ in our set $D_n$, take some arbitrary element of $B_n \subset C_nK_n$. For the next element in $D_n$, we cannot choose some other arbitrary element from $B_n$, because it might be the case that its shell intersects with the shell around $d_1$ (so $D_n$ would not satisfy the DSC). 
        %\revision{}{(line break)} 
        %\km{I think the continuity of ideas here suggests that we should save the line break until later.}
        By Lemma \ref{pos:trans_overlap} however, we know that for any $c \in C_nK_n$, we have $d_1K_n^{-1}K_n \cap cK_n^{-1}K_n \ne \varnothing \iff c \in d_1K_n^{-1}K_n(K_n^{-1}K_n)^{-1} = d_1(K_n^{-1}K_n)^2$. As a result, if we choose some element in $B_n \setminus d_1(K_n^{-1}K_n)^2$, we are guaranteed that its shell is disjoint from $d_1$. As long as we continue to remove the at most $|(K_n^{-1}K_n)^2|$ elements in $B_n$ which may have a shell intersecting with the latest element added to $D_n$, the sets we construct will always satisfy the DSC. Since we can repeat this process at least
        \begin{equation*}
            \left \lceil \frac{|B_n|}{|(K_n^{-1}K_n)^2|} \right \rceil
        \end{equation*}
        times, we can always construct $D_n$ satisfying the DSC and such that
        \begin{equation} \label{eqn:bumble}
            |D_n| \ge \frac{|B_n|}{|(K_n^{-1}K_n)^2|} \ge \frac{|B_n|}{|K_n^{-1}K_n|^2}.
        \end{equation}
        
        Now let
        \begin{equation*}
            H = (1 - \epsilon)\frac{2}{H_2(p)},
        \end{equation*}
        and note that $|K^{-1}_nK_n| \le H\ln(|C_nK_n|)$ for all large $n$ by \eqref{assump:shell_size_bound}. Then by applying this estimate in \eqref{eqn:bumble}, we see that
        \begin{equation} \label{eqn:bulldog}
            |D_n| \ge \frac{|B_n|}{H^2\ln(|C_nK_n|)^2}.
        \end{equation}
        
        Choose any $\delta > 0$ such that $\delta < \min(1,\epsilon)$. By \eqref{assump:B_large}, we have
        \begin{equation*}
            \lim_{n \to \infty} \frac{\ln(|B_n|)}{\ln(|C_nK_n|)} = 1,
        \end{equation*}
        and therefore, for all large enough $n$, we have $|B_n| \geq |C_n K_n|^{1-\delta/2}$.
        
        Combining this estimate with \eqref{eqn:bulldog}, we have
        \begin{equation*}
            |D_n| \ge \frac{|C_nK_n|^{1-\delta/2}}{H^2\ln(|C_nK_n|)^2}.
        \end{equation*}
        Furthermore, by \eqref{assump:B_few_shells}, we have
        \begin{equation*}
            \lim_{n \to \infty} \frac{\ln(|I_{B_n}|)}{\ln(|C_nK_n|)} = 0,
        \end{equation*}
        and therefore for all large enough $n$, we get $|I_{B_n}| \leq |C_n K_n|^{\delta/2}$. 
        
        Since $D_n \subset B_n$, it must be the case that 
        \begin{equation} \label{eqn:juniper}
            |I_{D_n}| \le |I_{B_n}| \leq |C_n K_n|^{\delta/2}.
        \end{equation}
        Now consider the set $U \subset D_n$ consisting of elements $a \in D_n$ that are the unique representatives of their respective shell types within $D_n$, i.e., for all $b \in D_n$, if $b \neq a$, then the shell type of $a$ is distinct from the shell type of $b$. Note that there is an injection from $U$ into $I_{D_n}$, given by mapping $a \in U$ to its shell type. Thus, $|U| \leq |I_{D_n}|$, and then by \eqref{eqn:juniper}, $|U| \leq |C_n K_n|^{\delta/2}$.  
        
        By replacing $D_n$ with the set $D_n \setminus U$ if necessary, we can guarantee that there are at least $2$ elements with any shell type represented in $D_n$. To summarize, for all large enough $n$, we have a set $D_n$ satisfying the DSC, having at least two elements of any shell type represented in $D_n$, and satisfying
        \begin{align} \begin{split} \label{eqn:vase}
            |D_n| &\ge \frac{|C_nK_n|^{1-\delta/2}}{H^2\ln(|C_nK_n|)^2} - |C_nK_n|^{\delta/2} \\
            &= |C_nK_n|^{\delta/2}\left[\frac{|C_nK_n|^{1-\delta}}{H^2\ln(|C_nK_n|)^2} - 1\right] \\
            &= |C_nK_n|^{\delta/2}\frac{|C_nK_n|^{1-\delta} - H^2\ln(|C_nK_n|)^2}{H^2\ln(|C_nK_n|)^2}.
            \end{split}
        \end{align}
        We also want $D_n$ to be nonempty, and the bound above may be negative. However, for large enough $n$, this bound is positive (in fact it tends to infinity). 
        
        We have thus verified all of the hypotheses of Lemma \ref{neg:id_prob} for all large enough $n$. Then by an application of this lemma, for all large enough $n$ we have
        \begin{equation} \label{eqn:grayandgreen}
            \mu_p^n(\RS_{C_n,K_n}) \ge 1 - \frac{4\sqrt{|I_{D_n}|}(1-\pi_2(p))^{-1}}{|D_n|\pi_2(p)^{\frac{1}{2}|K_n^{-1}K_n|}}\left[\frac{\sqrt{|I_{D_n}|}}{|D_n|\pi_2(p)^{\frac{1}{2}|K_n^{-1}K_n|}} + 2\right].
        \end{equation}
        Focusing for the moment on $\frac{\sqrt{|I_{D_n}|}}{|D_n|}$, we note that  $\sqrt{|I_{D_n}|} \le |I_{D_n}| \le |C_nK_n|^{\delta/2}$ by \eqref{eqn:juniper}, and then by \eqref{eqn:vase} we get 
        \begin{equation*}
            \frac{\sqrt{|I_{D_n}|}}{|D_n|} \le \frac{H^2\ln(|C_nK_n|)^2}{|C_nK_n|^{1 - \delta} - H^2\ln(|C_nK_n|)^2}.
        \end{equation*}
        
        Fixing any $\alpha >1$ and taking $n$ sufficiently large, the denominator in the above display is at least $|C_n K_n|^{1-\delta}/\alpha$, and therefore
        \begin{equation*}
            \frac{\sqrt{|I_{D_n}|}}{|D_n|} \le \frac{\alpha H^2\ln(|C_nK_n|)^2}{|C_nK_n|^{1 - \delta}} = \frac{\alpha H^2\ln(|C_nK_n|)^2|C_nK_n|^\delta}{|C_nK_n|}.
        \end{equation*}
        
        Combining this estimate with \eqref{eqn:grayandgreen} (and omitting the dependence on $n$ for the moment to aid readability), we find
        \begin{multline} \label{eqn:harper}
            \mu_p( \RS_{C,K} ) \ge \\ 
              1 - \frac{4\alpha H^2\ln(|CK|)^2|CK|^\delta(1-\pi_2(p))^{-1}}{|CK|\pi_2(p)^{\frac{1}{2}|K^{-1}K|}} \left[\frac{\alpha H^2\ln(|CK|)^2|CK|^\delta}{|CK|\pi_2(p)^{\frac{1}{2}|K^{-1}K|}} + 2 \right].
        \end{multline}
        Furthermore, by \eqref{assump:shell_size_bound}, we see that
        \begin{align*}
            |CK|\pi_2(p)^{\frac{1}{2}|K^{-1}K|} &\ge |CK|\pi_2(p)^{\frac{1}{2}(1-\epsilon)\frac{2}{H_2(p)}\ln(|CK|)} \\
            &= e^{\ln(|CK|)}e^{-(1-\epsilon)\ln(|CK|)} = |CK|^\epsilon.
        \end{align*}
        Combining the two previous displays gives
        \begin{equation*}
            \mu_p(\RS_{C,K}) \ge 1 - \frac{4\alpha H^2\ln(|CK|)^2(1-\pi_2(p))^{-1}}{|CK|^{\epsilon-\delta}} \left[\frac{\alpha H^2\ln(|CK|)^2}{|CK|^{\epsilon-\delta}} + 2\right].
        \end{equation*}
        By applying this estimate for all large enough $n$ and using that $\delta < \epsilon$ and $\mu_p^n(\RS_{C_n,K_n}) \leq 1$, we may let $n$ tend to infinity and conclude that
        \begin{equation*}
            \lim_{n \to \infty} \mu^n_p(\RS_{C_n,K_n}) = 1.
        \end{equation*}
        Then by Lemma \ref{neg:insct_bound}, we have %\revision{$e^x$}{$\exp(x)$}
        \begin{align*}
            \lim_{n \to \infty} \mu^n_p(\RS_{C_n,K_n} &\cap \cI_{C_n,K_n}) \le \lim_{n \to \infty} |C_nK_n|^3\pi_2(p)^{\frac{|C_nK_n|}{2}-1} \\
            &= \lim_{n \to \infty} \exp\left(3\ln(|C_nK_n|)-\left(\frac{|C_nK_n|}{2}-1\right)H_2(p)\right) \\
            &= \exp\bigl(H_2(p)\bigr) \lim_{n \to \infty} \exp\left(3\ln(|C_nK_n|) - \frac{H_2(p)}{2}|C_nK_n|\right) = 0.
        \end{align*}
        
        To conclude the argument, note that
        \begin{align*}
            0 \leq \mu^n_p(\cI_{C_n,K_n}) &= 1 - \mu^n_p(\cA^G\setminus\cI_{C_n,K_n}) \\
            &\le 1 - \mu^n_p(\RS_{C_n,K_n}\setminus\cI_{C_n,K_n}) \\
            &= 1 - \mu^n_p(\RS_{C_n,K_n}) + \mu^n_p(\RS_{C_n,K_n} \cap \cI_{C_n,K_n}).
        \end{align*}
        Then by the previous three displays, we see that
        \begin{align*}
            0 & \leq \lim_{n \to \infty} \mu^n_p(\cI_{C_n,K_n}) \\
            &\le 1 - \lim_{n \to \infty} \mu^n_p(\RS_{C_n,K_n}) + \lim_{n \to \infty} \mu^n_p(\RS_{C_n,K_n}\cap \cI_{C_n,K_n}) \\
            &= 1 - 1 + 0  = 0,
        \end{align*}
        which establishes that non-identifiability holds a.a.s.
    \end{proof}
\end{repthm}

%%%%%%%%%%%%%%%%%%%%%%%%%%%%%%%%%%%%%%%%%%%%%%
%%%%%%%%%%%%%%%%%%%%%%%%%%%%%%%%%%%%%%%%%%%%%%
\section{Examples} \label{sect:ex}
%%%%%%%%%%%%%%%%%%%%%%%%%%%%%%%%%%%%%%%%%%%%%%
%%%%%%%%%%%%%%%%%%%%%%%%%%%%%%%%%%%%%%%%%%%%%%

In this section we apply Theorems \ref{pos:regime} and \ref{neg:regime} to some families of examples and relate our results to the relevant previous work. 

%%%%%%%%%%%%%%%%%%%%%%%%%%%%%%%%%%%%%%%%%%%%%%%%%%%%%%%%%%%%%%%%%%%%%%%%%%%%%%%%%%%%
\subsection{Hypercubes in $\Z^d$ with sparse center sets} \label{Sect:LatticesBasic}
%%%%%%%%%%%%%%%%%%%%%%%%%%%%%%%%%%%%%%%%%%%%%%%%%%%%%%%%%%%%%%%%%%%%%%%%%%%%%%%%%%%%

Our first example is the main one analyzed in previous works: the $d$-dimensional integer lattice with both the center set and the read shape being hypercubes. In the $1$-dimensional case, patterns are just strings of characters, which is precisely how DNA sequences are modeled. The shotgun reconstruction problem on DNA sequences was analyzed in \cite{motahari}, with the identification problem on DNA sequences being studied in \cite{arratia}. In all other dimensions, this problem was analyzed by Mossel and Ross \cite{mossel}. 

We consider $\cA$, $p$, $d$, $\{r_n\}$, $\{m_n\}$, $\{\ell_n\}$, $\{R_n\}$, $\{G_n\}$, $\{C_n\}$, and $\{K_n\}$ as in Section \ref{sect:ex1}. Recall Corollary \ref{ex1}.

\begin{repcor}{ex1}
    If there exists some $\epsilon > 0$ such that for large enough $n$,
    \begin{equation*}
        r_n^d \ge (1 + \epsilon) \frac{2d}{H_2(p)}\ln(R_n),
    \end{equation*}
    then identifiability occurs a.a.s. On the other hand, if there exists some $\epsilon > 0$ such that for large enough $n$,
    \begin{equation*}
        r_n^d \le (1 - \epsilon) \frac{d}{2^{d-1}H_2(p)}\ln(R_n),
    \end{equation*}
    then non-identifiability occurs a.a.s.
\end{repcor}

%The following proof of the identifiability portion of Corollary \ref{ex1} is a straightforward application of our general identifiability result, Theorem \ref{pos:regime}.

\begin{proof}[Proof of Corollary \ref{ex1}, Identifiability]
    Assume the first display of the corollary. Consider two elements $c_1, c_2 \in C_n$ such that $c_1 = c_2 (\ell_n \cdot e_i)$ or $c_1 = c_2(\ell_n \cdot e_i^{-1})$ for some unit vector $e_i \in G_n$. The reads at these locations will overlap in the form of a $d$-dimensional rectangular prism, where each side has length $r_n$ except in the dimension $i$, where the length will be $r_n - \ell_n$ (which is positive for large enough $n$). Define the sequence $\{\bF_n\}$ by letting $\bF_n$ be the collection of all $d$ such $r_n^{d-1}(r_n-\ell_n)$ prisms (one for the choice of the shorter side being along each of the $d$ dimensions) for all large enough $n$. Let us now verify the hypotheses of Theorem \ref{pos:regime}.
    
    First, $|\bF_n| = d$ for large enough $n$, and therefore
    \begin{equation*}
        \lim_{n \to \infty} \frac{\ln(|\bF_n|)}{\ln(|C_nK_n|)} = 0.
    \end{equation*}
    This verifies \eqref{assump:few_oshapes}. Next, our choice of $\bF_n$ makes the overlap graph $\cO(G_n,C_n,K_n,\bF_n)$ connected, since it contains a square lattice as a subgraph, which is trivially connected. Thus \eqref{assump:overlap_conn} is satisfied. Finally, by assumption, we have that $r_n$ tends to infinity because $R_n$ does, and
    \begin{equation*}
        r_n^{d-1}(r_n-\ell_n) = \frac{r_n-\ell_n}{r_n}r_n^d \ge \frac{r_n-\ell_n}{r_n}(1 + \epsilon)\frac{2d}{H_2(p)} \ln(R_n). 
    \end{equation*}
    We then have that $\frac{r_n-\ell_n}{r_n}$ tends to 1 from below, since $\ell_n /r_n \to 0$  by assumption. As a result, for large enough $n$, we have
    \begin{equation*}
        \frac{r_n-\ell_n}{r_n} \ge 1 - \frac{\epsilon}{\epsilon+2} > 0,
    \end{equation*}
    and therefore
    \begin{align*}
        r_n^{d-1}(r_n-\ell_n) &\ge \left(1-\frac{\epsilon}{\epsilon+2}\right)(1+\epsilon)\frac{2d}{H_2(p)}\ln(R_n) \\
        &= \left(1 + \frac{\epsilon}{\epsilon+2}\right)\frac{2}{H_2(p)}\ln(R_n^d).
    \end{align*}
    Since $|C_nK_n| = R_n^d$ and $\forall F \in \bF_n$, $|F| = r_n^{d-1}(r_n-\ell_n)$, this inequality guarantees that \eqref{assump:oshape_size_bound} holds. Then by Theorem \ref{pos:regime}, identifiability occurs a.a.s.
\end{proof}

%Next, we give a proof of the non-identifiability portion of Corollary \ref{ex1}.

\begin{proof}[Proof of Corollary \ref{ex1}, Non-Identifiability]
    Assume the second display of the corollary. First, note that $|K_n^{-1}K_n| = (2r_n-1)^d \le (2r_n)^d$. Let us construct a sequence $B_n \subset C_nK_n$ and verify the hypotheses of Theorem \ref{neg:regime}. The restriction on $r_n^d$ imposed by the assumed inequality gives that $r_n \le L_0 \sqrt[d]{\ln(R_n)}$ for some constant $L_0$, and therefore $r_n/R_n \to 0$. As a result, we must have that $R_n \ge 2r_n - 1$ for all large enough $n$, so let $B_n = [r_n-1, R_n - r_n]^d$ (which is non-empty). This makes it so that $bK^{-1}_nK_n \subset C_nK_n$ for any $b \in B_n$, so none of the shells around any of the elements in $B_n$ \dquote{hang off the side}, which could make the shell around such an element abnormal in comparison (especially since $C_n$ is a regular lattice).
    
    Due to the regular nature of $C_n$, it is possible to place a hypercube with side length $\ell_n$ at each of the elements of $C_n$, and it will ultimately tile nearly all of $C_nK_n$ without overlap. $B_n$ is entirely contained within this tiling however. As a result, it is possible to uniquely represent each element $b \in B_n$ by some $c \in C_n$ and $a \in [0, \ell_n-1]^d \subset G_n$ where $b = ca$. We remark that we are using multiplicative notation here for consistency with the rest of the paper, although for this particular group, additive notation is more standard (e.g., $b = c + a$). Additionally, if some other $b' \in B_n$ is represented as $b' = c'a$ for some $c' \in C_n$ and the same $a \in [0, \ell_n-1]^d$, then $b$ and $b'$ have the same shell type, because they are both positioned identically relative to the underlying grid in $C_n$. As a result, there are at most $\ell_n^d$ different possible shell types in $B_n$, making $|I_{B_n}| = \ell_n^d$. 
    
    Since $\ell_n$ grows more slowly than $r_n$, for large enough $n$, we have $\ell_n < r_n$. Additionally, we note that the assumed inequality gives that $r_n \leq L_0 \ln(R_n)$ for a constant $L_0$ and all large enough $n$. By applying these inequalities, we get 
    \begin{equation*}
        \ln(\ell_n^d) \le \ln(r_n^d) \le d \ln(L_0\ln(R_n)) = d \ln(L_0) + d \ln( \ln(R_n)), 
    \end{equation*}
    and as a result, for all large enough $n$, we have
    \begin{align*}
        0 \leq \frac{\ln(|I_{B_n}|)}{\ln(|C_nK_n|)} &=  \frac{\ln(\ell_n^d)}{\ln(R_n^d)} \\
        &\leq  \frac{\ln(L_0)+ \ln(\ln(R_n))}{\ln(R_n)}.
    \end{align*}
    Taking the limit as $n$ tends to infinity, we obtain that 
    \begin{equation*}
        \lim_{n \to \infty} \frac{\ln(|I_{B_n}|)}{\ln(|C_nK_n|)} = 0,
    \end{equation*}
    which verifies \eqref{assump:B_few_shells}.
    
    Now recall that $|B_n| = (R_n - 2r_n + 2)^d$ and $|C_nK_n| = R_n^d$. Then with the fact that $r_n \ge 1$, we have
    \begin{align*}
        \lim_{n \to \infty} \frac{\ln(|B_n|)}{\ln(|C_nK_n|)} &= \lim_{n \to \infty} \frac{d\ln(R_n-2r_n+2)}{d\ln(R_n)} \\
        &= \lim_{n \to \infty} \frac{\ln(R_n)}{\ln(R_n)} + \frac{\ln\left(1 - 2\frac{r_n}{R_n} + \frac{2}{R_n}\right)}{\ln(R_n)} \\
        &= 1 + \lim_{n \to \infty} \frac{\ln\left(1 - 2\frac{r_n}{R_n} + \frac{2}{R_n}\right)}{\ln(R_n)} \\
        & = 1,
    \end{align*}
    where we have used the hypothesis that $r_n \leq  L_0\ln(R_n)$ in the last equality. This verifies \eqref{assump:B_large}.
    
    Additionally, by hypothesis, we have that
    \begin{equation*}
        r_n^d \le (1-\epsilon)\frac{d}{2^{d-1}H_2(p)}\ln(R_n),
    \end{equation*}
    and thus
    \begin{equation*}
        (2r_n-1)^d \le (1-\epsilon)\frac{2}{H_2(p)}\ln(R_n^d).
    \end{equation*}
    Since $|K_n^{-1}K_n| = (2r_n-1)^d$ and $|C_nK_n| = R_n^d$, we have verified condition \eqref{assump:shell_size_bound}. Then an application of Theorem \ref{neg:regime} gives that non-identifiability occurs a.a.s.
\end{proof}

Previous papers have analyzed the specific case when $\ell_n = 1$: when $d = 1$, it was treated in \cite{arratia, motahari}, and when $d \ge 2$ it was addressed in \cite{mossel}. Looking specifically at the case when $\ell_n = 1$ and $d = 1$, Corollary \ref{ex1} gives identifiability a.a.s. when 
\begin{equation*}
    r_n \ge (1 + \epsilon)\frac{2}{H_2(p)}\ln(R_n), 
\end{equation*}
and it gives non-identifiability a.a.s. when 
\begin{equation*}
    r_n \le (1 - \epsilon)\frac{1}{H_2(p)}\ln(R_n).
\end{equation*}
We note that this identifiability result coincides precisely with the identifiability threshold established in previous work \cite{arratia,motahari} when $d = 1$ and $\ell = 1$. On the other hand, our non-identifiability bound is smaller than the existing results for $d=1$ and $\ell=1$ by a factor of 2 \cite{motahari}. % relative to the existing results \cite{motahari} when $d=1$ and $\ell = 1$. 
The main reason for this discrepancy is that the previous result %\revision{s}{(remove s, aren't we only talking about one previous result here? I may be wrong however.)} 
involves a detailed analysis of DNA sequences, in which case one can exploit the specific structure of $\Z$ to consider more general blocking configurations and obtain the better bound. %Namely, they were able to analyze a broader class of blocking configurations in the case of $\Z$. 

In the case where $d$ is arbitrary and $\ell_n = 1$, Mossel and Ross \cite{mossel} have previous results similar to ours. In the case that the probability distribution $p$ is uniform, we exactly recover the results of \cite[Theorem 1.1]{mossel}. This can be seen by assuming $|\cA| = q$, which gives $H_2(p) = \ln(q)$ for uniform $p$. Additionally, when $p$ is not uniform, we recover precisely the same identifiability results. We note that our non-identifiability result is more general, in the sense that it covers the case of non-uniform $p$ while the previous work does not.

%%%%%%%%%%%%%%%%%%%%%%%%%%%%%%%%%%%%%%%%%%
\subsection{General read shapes in $\Z^d$}
%%%%%%%%%%%%%%%%%%%%%%%%%%%%%%%%%%%%%%%%%%

Let $\cA$, $p$, $d$, $\{G_n\}$, $\{C_n\}$, and $\{K_n\}$ be as defined in Section \ref{sect:ex2}. Recall the following corollary.

\begin{repcor}{ex2}
    If
    \begin{equation*}
        \lim_{n \to \infty} \frac{|\opint_1K_n|}{|K_n|} = 1,
    \end{equation*}
    and for large enough $n$,
    \begin{equation*}
        |K_n| \ge \frac{4d}{H_2(p)}\ln(n),
    \end{equation*}
    then identifiability occurs a.a.s. On the other hand, if
    \begin{equation*}
        \lim_{n \to \infty} \frac{\diam(K_n)}{n} = 0,
    \end{equation*}
    and for some $\epsilon > 0$ and large enough $n$,
    \begin{equation*}
        |K_n^{-1}K_n| \le (1 - \epsilon) \frac{2d}{H_2(p)} \ln(n),
    \end{equation*}
    then non-identifiability occurs a.a.s.
\end{repcor}

\begin{proof}[Proof of Corollary \ref{ex2}, Identifiability]
    To establish the identifiability statement, we assume that the first two displays of the corollary hold. We will verify the conditions of Theorem \ref{pos:regime}. First, by definition of $C_n$, we have that \eqref{assump:CK_to_inf} holds. 
    
    Next, define the sequence $\{F_n\}$ as $F_n = \opint_1 K_n$, and define the sequence $\{\bF_n\}$ as $\bF_n = \{F_n\}$ (the set containing $F_n$). Clearly, $|\bF_n| = 1$, and thus \eqref{assump:few_oshapes} holds.
    
    Now consider the sequence $\{\cO_n\}$ of overlap graphs $\cO_n = \cO(G_n,C_n,K_n,\bF_n)$. Consider any two elements $c_1, c_2 \in C_n$ that differ by a unit vector in $\Z^d$. We assume that $c_2 = c_1e_i$ for the standard basis %\revision{unit}{standard basis} 
    vector $e_i$ (such that $\| e_i \|_2 = 1$) without loss of generality (as the argument is symmetric for any other unit vector). We remark that while we are using multiplicative notation for the group, this actually means $c_2 = c_1 + e_i$ in the standard additive notation. We then note that
    \begin{equation*}
        c_1K_n \cap c_2K_n = c_1K_n \cap c_1e_iK_n = c_1\bigl(K_n \cap e_iK_n\bigr).
    \end{equation*}
    Since %\revision{$\| e_i \|_2 = 1$ by definition}{$e_i$ is a standard basis vector}
    $e_i$ is a standard basis vector, we have that $\opint_1K_n \subset K_n \cap e_iK_n$, and therefore
    \begin{equation*}
        c_1F_n = c_1\bigl(\opint_1 K_n\bigr) \subset c_1 \bigl(K_n \cap e_iK_n\bigr) = c_1K_n \cap c_2K_n.
    \end{equation*}
    Hence $(c_1,c_2) \in E(\cO_n)$. Since this is true for any two centers that differ by a unit vector, a regular lattice appears as a subgraph of $\cO_n$, and thus $\cO_n$ is connected, which verifies \eqref{assump:overlap_conn}.
    
    Lastly, let us show that for some $\epsilon > 0$ and large enough $n$,
    \begin{equation*}
        |F_n| \ge (1 + \epsilon) \frac{2}{H_2(p)}\ln(|C_nK_n|).
    \end{equation*}
    To accomplish this, let $\epsilon \in (0, \sfrac{1}{3})$, and define the constant $H = 2 (1+\epsilon) / H_2(p)$. First, we have that $|C_nK_n| \le |C_n||K_n|$, which means
    \begin{align} \begin{split} \label{eqn:walking}
        H\ln(|C_nK_n|) &\le H\ln(|C_n||K_n|) \\
        &= H\ln(|C_n|) + H\ln(|K_n|) \\
        &= Hd\ln(n) + H\ln(|K_n|),
        \end{split}
    \end{align}
    where we have used that $|C_n| = n^d$. Then by assumption, for all large enough $n$,
    \begin{equation*}
        |\opint_1 K_n| \ge (1-\epsilon)|K_n|.
    \end{equation*}
    
    Also, since $\ln(|K_n|)/|K_n| \to 0$, it must be the case that for large enough $n$,
    \begin{equation*}
        H\ln(|K_n|) \le \frac{1-3\epsilon}{2}|K_n|.
    \end{equation*}
    Applying this bound in \eqref{eqn:walking} gives that for all large enough $n$,
    \begin{align} \begin{split} \label{eqn:running}
        H\ln(|C_nK_n|) &\le Hd\ln(n) + H\ln(K_n) \\
        &\le Hd\ln(n) + \frac{1-3\epsilon}{2}|K_n|.
        \end{split}
    \end{align}
    Next, by assumption, we have for large enough $n$ that
    \begin{equation*}
        Hd\ln(n) \le \frac{1+\epsilon}{2}|K_n|.
    \end{equation*}
    Then by applying this bound in \eqref{eqn:running}, we see that for all large enough $n$, 
    \begin{align*}
        H\ln(|C_nK_n|) &\le Hd\ln(n) + \frac{1-3\epsilon}{2}|K_n| \\
        &\le \frac{1+\epsilon}{2}|K_n| + \frac{1-3\epsilon}{2}|K_n| \\
        &= (1-\epsilon)|K_n| \\
        &\le \opint_1 |K_n| \\
        &= |F_n|,
    \end{align*}
    which verifies \eqref{assump:oshape_size_bound}. We have shown that the hypotheses of Theorem \ref{pos:regime} are satisfied. By applying the theorem, we conclude that identifiability occurs a.a.s. 
\end{proof}

Next, we give the proof of the non-identifiability conditions for Corollary \ref{ex2}. For this, we remark that for any finite set $F$, we have that $\exists g \in G_n$ such that $F \subset g[0,\diam(F)]^d$, meaning that some axis-aligned hypercube with side length $\diam(F)$ contains $F$.

\begin{proof} [Proof of Corollary \ref{ex2}, Non-Identifiability]
    Assume the last two displays of the corollary hold. Let us verify the hypotheses of Theorem \ref{neg:regime}. Clearly \eqref{assump:CK_to_inf} holds. Next, for some $g \in G_n$, we have that $K_n \subset g[0, \diam(K_n)]^d$, and therefore $C_nK_n \subset g[0, n - 1 + \diam(K_n)]^d$ and $|C_nK_n| \le (n + \diam(K_n))^d$. Let $k_0 \in K_n$ be arbitrary, and define
    \begin{equation*}
        B_n = k_0[\diam(K_n), n-1-\diam(K_n)]^d.
    \end{equation*}
    Since $G_n$ is Abelian, $k_0 \in K_n$ and $[\diam(K_n), n-1-\diam(K_n)]^d \subset [0,n-1]^d = C_n$, we see that $B_n \subset C_nK_n$. By assumption, we have that $n$ grows faster than $\diam(K_n)$, and thus for large enough $n$, we have $n \ge 2\diam(K_n) + 1$. Hence $B_n$ is non-empty. 
    
    Let us now show that the shell around any element of $B_n$ is the maximal shell. To that end, let $b \in B_n$. By definition of $B_n$, there exists $l \in [\diam(K_n), n-1-\diam(K_n)]^d$ such that $b = k_0l$, and by definition of $g$, there exists $h_0 \in [0, \diam(K_n)]^d$ such that $k_0 = gh_0$. Putting these together, we get $b = k_0 l = g h_0 l$. Now let $k^{-1} \in K_n^{-1}$. Since $K_n^{-1} \subset g^{-1}[-\diam(K_n), 0]^d$, we have $k^{-1} = g^{-1}h_1^{-1}$ for some $h_1 \in [0, \diam(K_n)]^d$. Then since $G_n$ is Abelian, we have
    \begin{equation*}
        bk^{-1} = gh_0lg^{-1}h_1^{-1} = h_0lh_1^{-1}.
    \end{equation*}
    Next observe that
    \begin{equation*}
        [0, \diam(K_n)]^d[\diam(K_n), n-1-\diam(K_n)]^d[-\diam(K_n), 0]^d = [0,n-1]^d,
    \end{equation*}
    and therefore $bk^{-1} = h_0h_1^{-1}l \in [0,n-1]^d = C_n$. Since $k^{-1} \in K_n^{-1}$ was arbitrary, we have $bK^{-1}_n \subset C_n$, and therefore the shell around $b$ is the maximal shell around $b$. Thus there is only one shell type represented in $B_n$, i.e., $|I_{B_n}| = 1$. Then clearly \eqref{assump:B_few_shells} holds. 
    
    Next, we have that $|B_n| = (n - 2\diam(K_n))^d$ and $|C_nK_n| \le (n + \diam(K_n))^d$, which gives
    \begin{align*}
        1 \geq \frac{\ln(|B_n|)}{\ln(|C_nK_n|)} &\ge 
        \frac{\ln\bigl((n-2\diam(K_n))^d\bigr)}{\ln\bigl((n+\diam(K_n))^d\bigr)} \\
        &= 
        \frac{\ln(n) + \ln\left(1 - \frac{2\diam(K_n)}{n}\right)}{\ln(n) + \ln\left(1 + \frac{\diam(K_n)}{n}\right)}.
    \end{align*}
    By our assumption that $\diam(K_n) / n \to 0$, we may take the limit as $n$ tends to infinity and obtain that \eqref{assump:B_large} holds.
    
    Finally, by assumption, there exists $\epsilon > 0$ such that for large enough $n$, we have
    \begin{align*}
        |K_n^{-1}K_n| & \le (1 - \epsilon) \frac{2d}{H_2(p)}\ln(n) \\
        &= (1 - \epsilon) \frac{2}{H_2(p)}\ln(|C_n|) \\
        &\le (1 - \epsilon) \frac{2}{H_2(p)}\ln(|C_nK_n|).
    \end{align*}
    This verifies \eqref{assump:shell_size_bound}. Then by Theorem \ref{neg:regime}, non-identifiability occurs a.a.s.
\end{proof}

%%%%%%%%%%%%%%%%%%%%%%%%%%%%%%%%%%%%%%
\subsection{Finitely generated groups}
%%%%%%%%%%%%%%%%%%%%%%%%%%%%%%%%%%%%%%

For our last example, we fix an infinite but finitely generated group $G$ with finite symmetric generating set $T$. Let $T_i$, $\gamma(n)$, $\cA$, $p$, $\{r_n\}$, $\{R_n\}$, $\{G_n\}$, $\{C_n\}$, and $\{K_n\}$ be defined as in Section \ref{sect:ex3}. Corollary \ref{ex3} is stated as follows.

\begin{repcor}{ex3}
    If there exists some $\epsilon > 0$ such that for large enough $n$,
    \begin{equation*}
        \gamma(r_n-1) \ge (1 + \epsilon) \frac{2}{H_2(p)}\ln(\gamma(R_n))
    \end{equation*}
    then identifiability occurs a.a.s. If there exists some $\epsilon > 0$ such that for large enough $n$,
    \begin{equation*}
        \gamma(2r_n) \le (1-\epsilon) \frac{2}{H_2(p)}\ln(\gamma(R_n)),
    \end{equation*}
    then non-identifiability occurs a.a.s.
\end{repcor}

To prove the identifiability statement of the corollary, we require the \textit{$T$-interior} of a set $F \subset G$, denoted by $\opint_T F$ and defined as
\begin{equation*}
    \opint_T F = \bigcap_{t \in T} F \cap tF.
\end{equation*}
Additionally, it is straightforward to show that for any $n,m \in \N$, we have $T_{n+m} = T_nT_m$, and therefore $\gamma(n+m) \le \gamma(n)\gamma(m)$, which is known as \textit{submultiplicativity}. %\revision{}{(added previous .)}
%\revision{}{, which is known as \textit{submultiplicativity}}. 
We now prove our identifiability result for this example.

\begin{proof} [Proof of Corollary \ref{ex3}, Identifiability]
    Assume the first display of the corollary. Let us verify the hypotheses of Theorem \ref{pos:regime}. To begin, we define the sequence $\{\bF_n\}$ as
    \begin{equation*}
        \bF_n = \{\opint_T K_n\}.
    \end{equation*}
    Clearly then, for any elements $c, ct \in C_n$ for any $t \in T$, the intersection of the reads will be $cK_n \cap ctK_n = c(K_n \cap tK_n)$, which is a translate of $K_n \cap tK_n$ that contains $\opint_T K_n$ by definition, so $c,ct$ will be connected by an edge in the overlap graph. This means that the Cayley graph $\Gamma = \Gamma(G,T)$ is a subgraph of the overlap graph $\cO_n = \cO(G_n,C_n,K_n,\bF_n)$, and since $\Gamma$ is connected (because $T$ is a symmetric generating set),  $\cO_n$ must be connected as well. Hence \eqref{assump:overlap_conn} holds. Also, we have that $|\bF_n| = 1$ and thus \eqref{assump:few_oshapes} clearly holds.
    
    Our goal now is to bound the size of $\opint_T K_n = \opint_T T_{r_n}$, which we will do by showing $T_{r_n-1} \subset \opint_T T_{r_n}$. Let $h \in T_{r_n-1}$ and $t \in T$, so we clearly also have that $h \in T_{r_n}$. We also have that $t^{-1}h \in T_{r_n}$, because $t^{-1} \in T$, so $t^{-1}h$ can be represented with at most $r_n$ generators. This gives $h = t(t^{-1}h) \in tT_{r_n}$, which means $h \in T_{r_n} \cap tT_{r_n}$ for all $t$, and $h \in \opint_T T_{r_n}$. Therefore, $T_{r_n-1} \subset \opint_T T_{r_n}$, and
    \begin{equation*}
        |\opint_T K_n| \ge |T_{r_n-1}| = \gamma(r_n-1).
    \end{equation*}
    Then by assumption, we have that there exists an $\epsilon > 0$ such that for large enough $n$, $\forall F \in \bF_n$,
    \begin{equation*}
        |F| = |\opint_T K_n| \ge \gamma(r_n-1) \ge (1+\epsilon)\frac{2}{H_2(p)}\ln(\gamma(R_n)) = (1+\epsilon)\frac{2}{H_2(p)}\ln(|C_nK_n|).
    \end{equation*}
    This shows that \eqref{assump:oshape_size_bound} holds. Then by Theorem \ref{pos:regime}, we conclude that identifiability occurs a.a.s.
\end{proof}

We now give our prove of the non-identifiability result for this example.

\begin{proof} [Proof of Corollary \ref{ex3}, Non-Identifiability]
    Assume the second display of the corollary. Let us verify the hypotheses of Theorem \ref{neg:regime}. First, we will establish that $\forall c > 0$, for all large enough $n$, we have $R_n > cr_n$. To see why this is the case, suppose for contradiction that $\exists c > 0$ such that $R_{n_k} \le cr_{n_k}$ for a subsequence $\{n_k\}$ tending to infinity. Then since $\gamma$ is monotonically increasing, we have
    
    \begin{equation*}
        \gamma\left(\frac{2}{c}R_{n_k}\right) \le (1 - \epsilon) \frac{2}{H_2(p)}\ln(\gamma(R_{n_k})).
    \end{equation*}
    If $\frac{2}{c} \ge 1$, then by the fact $\gamma$ is monotonically increasing, we have
    \begin{equation*}
        \gamma(R_{n_k}) \le \gamma\left(\frac{2}{c}R_{n_k}\right) \le (1 - \epsilon) \frac{2}{H_2(p)}\ln(\gamma(R_{n_k})),
    \end{equation*}
    which is clearly not true for large $k$, since $\gamma$ grows faster than $\ln \circ \gamma$. On the other hand, if $\frac{2}{c} < 1$, then let $x_{n_k} = \frac{2}{c}R_{n_k}$ and $m = \lceil c \rceil$. By the submultiplicativity of $\gamma$, note that we have $\gamma(nx) \le \gamma(x)^n$ for all $n,x \in \N$. By the previous display and this fact, we see that
    \begin{align*}
        \gamma\left(x_{n_k}\right) &\le (1 - \epsilon) \frac{2}{H_2(p)}\ln(\gamma(R_{n_k})) \\
        &\le (1 - \epsilon) \frac{2}{H_2(p)} \ln\left(\gamma\left(m\frac{2}{c}R_{n_k}\right)\right) \\
        &\le (1 - \epsilon) \frac{2}{H_2(p)} \ln(\gamma(x_{n_k})^m) \\
        &= (1-\epsilon) \frac{2m}{H_2(p)} \ln(\gamma(x_{n_k})).
    \end{align*}
    However, this inequality cannot hold for all large $k$, since $\gamma$ grows faster than $\ln \circ \gamma$. Thus, we have shown that $\forall c > 0$, for all large $n$, we must have $R_n > cr_n$, and hence $\lim_n r_n/R_n = 0$. We can then define the sequence $\{B_n\}$ as $B_n = T_{R_n-2r_n}$, which is non-empty for large enough $n$ (as $R_n - 2r_n$ must tend to infinity). 
    
    We start by counting the number of shell types represented in $B_n$, denoted $|I_{B_n}|$. Let $b \in B_n$, and consider $bK^{-1}_n = bT^{-1}_{r_n}$. Since $T_{r_n}$ is symmetric, we have that $T^{-1}_{r_n} = T_{r_n}$, which means $bK^{-1}_n = bT_{r_n}$. Let $h \in bT_{r_n}$, so for some $t \in T_{r_n}$, $h = bt$. Since $b \in T_{R_n-2r_n}$, we must have that $h \in T_{R_n-2r_n+r_n} = T_{R_n-r_n} = C_n$, so we have that $bK^{-1}_n \subset C_n$. Therefore, the shell around any element in $B_n$ is the maximal shell around that element, and there is a single shell type represented in $B_n$. This clearly yields \eqref{assump:B_few_shells}.
    
    Next, since $|B_n| = \gamma(R_n-2r_n)$, we have
    \begin{align*}
        1 \geq \frac{\ln(|B_n|)}{\ln(|C_nK_n|)} &=  \frac{\ln(\gamma(R_n-2r_n))}{\ln(\gamma(R_n))} \\
        &=  \frac{\ln(\gamma(R_n-2r_n)\gamma(2r_n))-\ln(\gamma(2r_n))}{\ln(\gamma(R_n))} \\
        & \ge  \frac{\ln(\gamma(R_n))-\ln(\gamma(2r_n))}{\ln(\gamma(R_n))} \\
        &= 1 -  \frac{\ln(\gamma(2r_n))}{\ln(\gamma(R_n))},
    \end{align*}
    where the inequality follows from the fact that $\gamma(n+m) \le \gamma(n)\gamma(m)$. By assumption, we have that $\gamma(2r_n) = L_0 \ln(\gamma(R_n))$ for some constant $L_0$, and therefore $\ln(\gamma(2r_n)) \leq L_1 + \ln( \ln(\gamma(R_n)))$ for some constant $L_1$. Applying this inequality in the previous display and taking the limit as $n$ tends to infinity gives 
    \begin{equation*}
        \lim_{n \to \infty} \frac{ \ln(|B_n|)}{\ln(|C_n K_n|)} = 1, 
    \end{equation*}
    and thus \eqref{assump:B_large} holds.
    
    Lastly, we have that
    \begin{equation*}
        |K^{-1}_nK_n| = |T^{-1}_{r_n}T_{r_n}| = |T_{r_n}T_{r_n}| = |T_{2r_n}| = \gamma(2r_n),
    \end{equation*}
    and then by assumption we conclude that
    \begin{equation*}
        |K^{-1}_nK_n| = \gamma(2r_n) \le (1 - \epsilon) \frac{2}{H_2(p)}\ln(\gamma(R_n)) = (1-\epsilon)\frac{2}{H_2(p)}\ln(|C_nK_n|),
    \end{equation*}
    which verifies \eqref{assump:shell_size_bound}. Then by Theorem \ref{neg:regime}, non-identifiability occurs a.a.s.
\end{proof}

%%%%%%%%%%%%%%%%%%%%%%%%%%%%%%%%%%%%%%%%%%%%%%%%%%%%%%%%%
%%%%%%%%%%%%%%%%%%%%%%%%%%%%%%%%%%%%%%%%%%%%%%%%%%%%%%%%%
\section{Final remarks and future work} \label{sect:conc}
%%%%%%%%%%%%%%%%%%%%%%%%%%%%%%%%%%%%%%%%%%%%%%%%%%%%%%%%%
%%%%%%%%%%%%%%%%%%%%%%%%%%%%%%%%%%%%%%%%%%%%%%%%%%%%%%%%%
 
 In this work we consider shotgun identification problems in the context of discrete groups. We present both positive and negative results and illustrate these results with several families of examples. In future work, we think the following general framework for shotgun identification problems may be interesting to consider.

\begin{prob}[Abstract Shotgun Identification]
    Let $S$ be a finite set, and let $\cA$ be a finite alphabet. Here $S$ serves as a base set to be labeled by symbols from $\cA$. We call such assignments \textit{patterns}, and denote the set of all such patterns by $\cA^S$. These patterns can be viewed as functions from $S$ to $\cA$. Let $\Sigma \subset \cP(S)$ be a collection of locations at which a pattern $w \in \cA^S$ will be observed. Let $\Omega$ be some \textit{observation space}, and for every $T \in \Sigma$ define an \textit{observation function} $\alpha_T:\cA^T \to \Omega$. With $\MS(X)$ representing the multi-sets of elements of a set $X$, the \textit{read operator} $\cR_\Sigma:\cA^S \to \MS(\Omega)$ is defined as
    \begin{equation*}
        \cR_\Sigma(w) = \bigl\{\alpha_T\bigl(x|_T\bigr) : T \in \Sigma\bigr\},
    \end{equation*}
    whose elements are called the \textit{reads} of a pattern $w \in \cA^S$. In other words, the read operator provides the observations of $w$ at each element $T \in \Sigma$. Note that the location of the observation in $w$ may not be known, but certain local information can remain intact based on the nature of the observation functions. Let $P \le \Sym(S)$ be a subgroup of the symmetric group on $S$. These permutations define a natural notion of equivalence on $\cA^S$, namely that two patterns $w$ and $x$ are equivalent if there is some $p \in P$ such that $x \circ p = w$. Letting $\cR_\Sigma^{-1}$ denote the set inverse of $\cR_\Sigma$, a pattern $x \in \cA^S$ is said to be \textit{identifiable} if
    \begin{equation*}
        \cR_\Sigma^{-1}\bigl(\cR_\Sigma(w)\bigr) = \{x \in \cA^S : \exists p \in P \tst x \circ p = w\}.
    \end{equation*}
    In other words, a pattern $w \in \cA^S$ is identifiable if given $S$, $P$, $\Sigma$, $\cA$, $\Omega$, $\{\alpha_T : T \in \Sigma \}$, and $\cR_\Sigma(w)$, it is possible to uniquely identify the original pattern in $\cA^S$ which produced the multi-set $\cR_\Sigma(w)$ up to a permutation $p \in P$. The main question in the shotgun identification problem is then, \textit{what makes a pattern identifiable?} Furthermore, if a probability measure is defined on $\cA^S$, then one may ask, \textit{how likely is it for a random pattern to be identifiable?}
\end{prob}

Our work gives insights into these questions in the case of groups, giving sufficient conditions for identifiability a.a.s. and sufficient conditions for non-identifiability a.a.s. However, there are many ways that this work can be extended, such as studying identifiability of structures other than groups or even tackling the shotgun identification problem at the level of generality stated here. In the same direction, since the set $\Sigma$ is arbitrary, the reads need not have the same shape, so it may be interesting to analyze the case where there are multiple read shapes ($K_1$, $K_2$, etc.).

On the probabilistic side of the problem, it may be interesting to analyze \dquote{noisy reads}, in which the observations may contain some noise. Although the presence of noisy reads clearly adds some complexity to the problem, we expect that some of the basic tools of this paper should still be useful. Our probabilistic analysis focuses on product measures on the group $G$, under which the symbols are chosen in an i.i.d. manner. Future work could involve analyzing different probability spaces on $\cA^G$, potentially adding dependence between the labels of two elements. While many of the combinatorial results of this paper would still apply, little to none of the probabilistic results are likely to be applicable.

The critical phenomenon observed in the analysis DNA sequence reconstruction by Motahari et al. \cite{motahari} leads to a natural question: \textit{does this critical phenomenon exist for shotgun identification problems more generally?} For shotgun identification problems on graphs, there are some results and conjectures in this direction \cite{mossel}. The work presented here suggests that such critical phenomena may exist for shotgun identification problems on groups as well. We believe it would be interesting to analyze the critical phenomenon in greater detail. In particular, is there a type of phase transition for shotgun identification problems in general, and if so, is $\lambda_c = \frac{2}{H_2(p)}$ the correct threshold? 

Lastly, let us mention that in the shotgun identification problems considered here, one has knowledge of the center set $C$ at the time of reconstruction. However, it may be interesting to study different observational paradigms. In particular, it might be interesting to consider the case in which the center set consists of some number of locations drawn at random from some sample distribution. 

%%%%%%%%%%%%%%%%%%%%%%%%%%
%%%%%%%%%%%%%%%%%%%%%%%%%%
\section*{Acknowledgements}
%%%%%%%%%%%%%%%%%%%%%%%%%%
%%%%%%%%%%%%%%%%%%%%%%%%%%

All authors gratefully acknowledge the support of the National Science Foundation via grant DMS 1847144. Additionally, the authors thank Kirsten Gilmore, Daniel Johnson, and Greg Wolcott for many helpful conversations about shotgun sequencing.

%%%%%%%%%%%%%%%%%%%%%%%%%
%%%%%%%%%%%%%%%%%%%%%%%%%
\bibliographystyle{abbrv}
\bibliography{refs}
%%%%%%%%%%%%%%%%%%%%%%%%%
%%%%%%%%%%%%%%%%%%%%%%%%%

\end{document}